\documentclass[11 pt,letterpaper] {amsart}
\usepackage {amssymb,latexsym,amsthm,amsmath,amsfonts, mathtools, multirow, longtable,comment,tikz-cd}
\usepackage{enumitem,color}
\usepackage[hidelinks]{hyperref}

\usepackage{hyperref}
\hypersetup{
	colorlinks,
	citecolor=blue,
	filecolor=black,
	linkcolor=red,
	urlcolor=black
}

\usepackage{mathdots}

\long\def\symbolfootnote[#1]#2{\begingroup
	\def\thefootnote{\fnsymbol{footnote}}\footnote[#1]{#2}\endgroup}
\newcommand{\Z}{\mathbb Z}
\newcommand{\Q}{\mathbb Q}

\newcommand{\C}{\mathbb C}

\newcommand{\mco}{\mathcal O}
\newcommand{\Ind}{\mathrm{Ind}}
\newcommand{\GL}{\mathrm{GL}}

\newcommand{\sym}{\mathrm{sym}}

\makeatletter
\def\imod#1{\allowbreak\mkern10mu({\operator@font mod}\,\,#1)}
\makeatother
\newtheorem{theorem}{Theorem}[section]
\newtheorem{lemma}[theorem]{Lemma}
\newtheorem{corollary}[theorem]{Corollary}
\newtheorem{proposition}[theorem]{Proposition}
\newtheorem*{theorem*}{Theorem}
\theoremstyle{definition}
\newtheorem{remark}[theorem]{Remark}

\numberwithin{equation}{section}

\newtheorem{defn}[subsection]{\bf Definition}

\title[Twists of epsilon factors for symmetric square transfers of a modular form]{On the change of epsilon factors for symmetric square transfers under twisting and applications}

\author{Tathagata Mandal and Sudipa Mondal}
\address{Tathagata Mandal, math.tathagata@gmail.com, ISI Kolkata, Kolkata 700108, India}
\address{Sudipa Mondal, sudipa.mondal123@gmail.com, HRI, Prayagraj 211019, India}

\begin{document}
	\begin{abstract}
		Let us consider the symmetric square transfer of the automorphic representation $\pi$ associated to a modular form $f \in S_k(N, \epsilon)$. In this article, we study the variation of the epsilon factor of  $\sym^2(\pi)$ under twisting in terms of the local Weil-Deligne representation at each prime $p$. As an application, we detect the possible types of the symmetric square transfer of the local representation at $p$.  Furthermore, as the conductor of $\sym^2(\pi)$ is involved in the variation number, we compute it in terms of $N$. %We also compute the conductor of $\sym^2(\pi)$ in terms of $N$ as it is involved in the variation number. %For $\sym^2$ transfer, the most difficult prime is $p=2$.
	\end{abstract}
	
	\maketitle
	
\section{Introduction}\label{introduction}
Let $f \in S_k(N, \epsilon)$ be a cuspidal newform of weight $k$, level $N$ and nebentypus $\epsilon$. Let $\pi_f$ be the automorphic representation of the ad\`ele group $\GL_2(\mathbb{A}_\Q)$ attached to $f$. This $\pi_f$ decomposes as a restricted tensor product
$
\pi_f=\bigotimes_v' \pi_{f,v},
$
where $v$ varies over all places of $\Q$. Here $\pi_{f,v}$ denotes an irreducible admissible representation of $\GL_2(\Q_v)$ and $\Q_v$ denotes the completion of $\Q$ at $v$.

\smallskip 

For a prime $p \mid N$, the local representation $\pi_{f,p}$ can be either of ramified principal series, special or supercuspidal type. Determining the type of $\pi_{f,p}$ is very important in many aspect. For example, the ramification formulae of the local endomorphism algebra attached to $f$ depends on these types, see \cite{bmjnt}. In another direction, Pacetti-Kohen found it useful while constructing Heegner/Darmon points \cite{kp18}. These types are determined in \cite{bm} and \cite{pacetti} when $f$ has trivial or arbitrary nebentypus respectively using the theory epsilon factors. Note that by the local inverse theorem \cite[Section 27]{bh06}, one can determine the types by the variation of the local epsilon factors by twisting with respect to a certain set of characters. An alternative approach is provided in \cite{lw} where the authors give an algorithm to determine the types of the local representations using the theory of modular symbols and the results of Casselman on the structure of irreducible representations of $\GL_2(\Q_p)$.

\smallskip  

In this present paper, we are interested in determining the types of the symmetric square transfer of $\pi_{f,p}$ in terms of the ratio of the epsilon factor of $\sym^2(\pi_f)$ with respect to twisting by a quadratic character ramified only at $p$. The conductor of $\sym^2(\pi_f)$  is involved in this variance number. In the second result, we determine this conductor in terms of $N$. For the statement of results of this article, we refer to \S \ref{statement}. We mention that the same has been determined for symmetric cube transfers in \cite{bmm}.

\smallskip

While computing the ratio of local epsilon factors, we use the same additive character and the Haar measure on both the components. Note that the supercuspidal types are the interesting one and we use an additive character of conductor zero in such cases. If $p$ is an odd prime such that $C_p<N_p\geq 2$ then the corresponding local representation attached to a $p$-minimal form becomes supercuspidal and it is a representation induced from a quadratic extension $K/\Q_p$. Depending upon $K/\Q_p$ unramified or ramified, we call that $p$ a unramified or ramified supercuspidal prime for $f$. In Corollaries \ref{maincoro} and \ref{mainkorop=2}, we classify the quadratic extensions $K/\Q_p$.

\smallskip

The paper is organized as follows. The conductor of characters, symmetric square transfer, epsilon factors and related results are discussed in section \ref{preli}. In section \ref{statement} we state the results of this paper. Section \ref{l-parameter} contains the $L$-parameters of local representations. We give the proofs for the non-supercuspidal representations in section \ref{non-sup}. In section \ref{sup} we discuss the results for the supercuspidal representations and classify the types. The last section deals with the conductor of the symmetric square transfer of $\pi_f$.

\smallskip
	
\noindent \textbf{Notation}: For a $p$-adic field $F$, let $W(F), W'(F)$ be its Weil group and Weil-Deligne group respectively. The field with $p^r$ elements is denoted by $\mathbb{F}_{p^r}$. The symbol $\big( \frac{.}{p}\big)$ is the usual Legendre symbol.	

\smallskip

\noindent \textit{Acknowledgment}: The first author acknowledges NBHM postdoctoral fellowship at ISI Kolkata. The second author is supported by the institute postdoctoral fellowship at HRI Prayagraj.

\section{Preliminaries}\label{preli}
\subsection{Characters and Gauss sum:}
For  a non-archimedean local field $F$ of characteristic zero, let $\mathcal{O}_F, \mathfrak{p}_F, \kappa_F$ denote the ring of integers, the maximal ideal and the residue field of $F$ respectively. For $n \geq 0$, let $U_F^n= 1+ \mathfrak{p}_F^n$ be the $n$-th principal units of $F$.
\begin{defn}
	Let $\chi$ be a multiplicative character of $F$. The conductor $a(\chi)$ of $\chi$ is the smallest positive integer $i$ such that $\chi_{|_{1+{\mathfrak{p}_F^i}}}=1$. For a non-trivial additive character $\phi$ of $F$, its conductor  is an integer $n(\phi)$ such that $\phi|_{\mathfrak{p}_K^{-n(\phi)}}=1$ but $\phi|_{\mathfrak{p}_K^{-n(\phi)-1}}\neq1$.

	The set of all multiplicative (resp. additive) characters of $F$ is denoted by $\widehat{F^\times}$ (resp. $\widehat{F}$). For $\chi_1, \chi_2 \in \widehat{F^\times}$, we have $a(\chi_1 \chi_2) \leq \text{max}\{a(\chi_1), a(\chi_2)\}$. 
	A character $\chi$ is said to be unramified if $a(\chi)=0$ and tamely ramified if $a(\chi)=1$. %{\color{blue} If $a(\chi)=t\geq 1$, then it induces a non-trivial character $\widetilde{\chi}: \mathcal{O}_{K}^\times/U_K^t \to \C^\times$.} Here $U_K^n= 1+ \mathfrak{p}_K^n$ is the $n$-th principal units of $K$. 
\end{defn}
\begin{proposition}\label{chisquare}
	Let $p$ be an odd prime and $\chi$ be a character of $\Q_p^\times$ with conductor $a(\chi)$. We have
	\[
	a(\chi^2) = \begin{cases}
		0  & \text{ if the order of } \widetilde{\chi} \text{ is } 2 \text{ with } a(\chi)=1 \\
		a(\chi) & \text{ otherwise. }
	\end{cases} \]
	For $p=2$, we have $a(\chi^2) =0$ if $a(\chi)=2, 3$, otherwise $a(\chi)-1$.
\end{proposition}

\begin{proof}
	When $p\geq 3$, the proof follows from the definition of conductors. For $p=2$, see \cite[Prop. 2.3]{bmm}.
\end{proof}

%\subsection{Gauss Sum} \label{Gauss} Let $\mathbb{F}_{p^r}$ be a field of order $p^r$. 
For $\chi \in \widehat{\mathbb{F}_{p^r}^\times}$ and $\phi \in \widehat{\mathbb{F}_{p^r}}$, the classical {\it Gauss sum} is defined by
$
G(\chi,\phi)=\sum_{x \in \mathbb{F}_{p^r}^\times}^{} \chi(x) \phi(x).
$
Denote it by $G_r(\chi)$ when $\phi$ is fixed. If $r=1$ and $\chi$ has order $k$, then following Gross-Koblitz formula we obtain the Gauss sum in terms of $p$-adic gamma function $\Gamma_p$ as follows: \cite[Corollary~$3.1$]{kmy}:
\begin{eqnarray} \label{GKcoro}
	G_1(\chi^a)=(-p)^{a/k} 
	\Gamma_p \big( \frac{a}{k} \big).
\end{eqnarray}

For $\chi \in \widehat{\mathbb{F}_p^\times}$ and 
$\phi \in \widehat{\mathbb{F}_p}$, let $\chi'=\chi \circ N_{\mathbb{F}_{p^r}|\mathbb{F}_p}$
and $\phi'=\phi \circ \mathrm{Tr}_{\mathbb{F}_{p^r}|\mathbb{F}_p}$ be their lifts
to $\mathbb{F}_{p^r}$. The Davenport-Hasse theorem \cite[Theorem~$11.5.2$]{MR1625181} says that
$G(\chi',\phi')=(-1)^{r-1}G(\chi,\phi)$.

\subsection{Symmetric square transfer:} \label{sym2} Let $K$ be a number field and $\mathbb{A}_K$ denotes the ring of ad\`eles. Let $\pi=\otimes'_{v \leq \infty}\pi_v$ be a cuspidal automorphic representation of ${\rm GL}_2(\mathbb{A}_K)$. The symmetric square ${\rm {sym}}^2: \text{GL}_2(\mathbb{C}) \to \text{GL}_3(\mathbb{C})$ of the standard representation of ${\rm GL}_2$ is defined as follows: 
\begin{eqnarray} \label{symsquare}
	\text{sym}^2 \left(\begin{bmatrix} a & b \\ c & d \end{bmatrix} \right) =  \begin{bmatrix} a^2 & ab & b^2 \\ 2ac & ad+bc & 2bd \\ c^2 & cd & d^2 \end{bmatrix}.
\end{eqnarray}

Let $\phi_v$ be the $2$-dimensional representation of the Weil-Deligne group $W'(K_v)$ attached to $\pi_v$. Then ${\rm sym^2}(\phi_v)= {\rm sym}^2\circ \phi_v$ is a $3$-dimensional representation of $W'(K_v)$. Let ${\rm sym}^2(\pi_v)$ be the irreducible admissible representation of ${\rm GL}_3(K_v)$ attached to ${\rm sym}^2(\phi_v)$ by the local Langlands correspondence. Set ${\rm sym}^2(\pi):=\otimes_v'{\rm sym}^2(\pi_v)$. This is known as the symmetric square transfer of $\pi$, see \cite{MR0533066} for more details.

\subsection{Epsilon factors}  \label{epsilon}
Let $\chi \in \widehat{F^\times}$ and $\phi \in \widehat{F}$ be two non-trivial characters.
The local $\varepsilon$-factor associated to them has the following description \cite[p. $94$]{MR0457408}:
\begin{equation}\label{defepsilonfactor}
\varepsilon(\chi, \phi, c) =q^{- \frac{a(\chi)}{2}} \chi(c) \tau(\chi, \phi),
\end{equation}
where $c \in F^\times$ has valuation $a(\chi)+n(\phi)$ and 
$\tau(\chi, \phi)=  \sum_{x \in \frac{\mco_F^\times}{U_F^{a(\chi)}}}^{} \chi^{-1}(x) \phi(\frac{x}{c})$ is
the local Gauss sum associated to $\chi$
and $\phi$. It can be shown that the above definition is independent of the choice of $c$; so we simply write $\varepsilon(\chi, \phi, c)=\varepsilon(\chi, \phi)$. If $\chi$ is unramified, then $c$ has the valuation $n(\phi)$ and so $\varepsilon(\chi,\phi,c)=\chi(c)$. For $a(\chi)=1$, note that $\widetilde{\chi}:=\chi^{-1}|_{\mco_F^\times} \in \widehat{\kappa_F^\times}$. In this case, choosing an additive character $\phi$ of conductor $-1$, we deduce that the local Gauss sum is the same as the well-known classical Gauss sum. Here we list some properties of epsilon factors  \cite{tate}.
\begin{enumerate}
	\item [($\epsilon 0$)] $\varepsilon(\chi, \phi_a)=\chi(a)|a|_F^{-1} \varepsilon(\chi,\phi)$, where $a \in F^\times$, $\phi_a(x)=\phi(ax)$ and $|\,\,|_F$ is the absolute value of $F$.
	\item  [($\epsilon 1$)]
	$\varepsilon(\chi\theta, \phi)=\theta(\pi_F)^{a(\chi)+n(\phi)}
	\varepsilon(\chi,\phi)$, where $\theta \in \widehat{F^\times}$ is unramified.
	%The element $\pi_F$ is a uniformizer of $F$.
	\item  [($\epsilon 2$)]
	$\varepsilon \big(\Ind_{W(F)}^{W(\Q_p)} \rho, \phi \big)
	=\varepsilon \big(\rho, \phi_F \big)$,
	where $\phi_F=\phi \circ \mathrm{Tr}_{F/\Q_p}$ and $\rho$ is a virtual zero dimensional representation 
	of a finite extension $F/\Q_p$.
\end{enumerate}

%\subsection{Computation of $\varepsilon$-factor:} 
We now describe how to compute the $\varepsilon$-factor of a representation $\rho'=(\rho, N')$ of $W'(\Q_p)$ (cf. \S \ref{statement}). Suppose $\rho'$ acts on the space $V$. Let $\Phi \in W(\Q_p)$ be an inverse Frobenius element and $I$ be the inertia subgroup of $W(\Q_p)$. Consider the following subspaces of $V$:
\[V^I=\{ v \in V : \rho(g) v =v \text{ for all } g \in I\}, \quad V_{N'} =\ker(N') \text{ and } V_{N'}^I=V^I\cap V_{N'}.\]
%\[V_{N'} =\ker(N') \text{ and } V_{N'}^I=V^I\cap V_{N'}.\]
\begin{proposition} \label{epsdef} \cite[Equ.(2.49)]{BR} The $\varepsilon$-factor of $\rho'$ is given by $\varepsilon(s, \rho', \phi) = \varepsilon(s, \rho, \phi) \det (-\rho(\Phi) p^{-s} | V^I/V^I_{N'}).$
\end{proposition}

Consider a $p$-minimal modular form $f \in S_k(N, \epsilon)$, that is, the $p$-part of its level is the smallest among all twists of $f$ by Dirichlet characters. Let $\pi:=\pi_f$ denote the automorphic representation attached to $f$ with $\pi_p:=\pi_{f,p}$ the local representation at $p$. In this article, we study the variation number 
\begin{eqnarray} \label{ratio}   \varepsilon_p:=\frac{\varepsilon \left(\sym^2(\pi_{p}) \otimes \chi_p \right)}{\varepsilon \left(\sym^2(\pi_{p}) \right)}\end{eqnarray}
depending upon the type of $\pi_p$ under twisting by a quadratic character $\chi_p$ which is defined as follows.
Consider the quadratic extension of $\Q$ ramified only at $p$ with its associated quadratic character $\chi$. We can identify $\chi$ with a character of the id\`{e}le group by class field theory, i.e., 
characters $\{\chi_q\}_q$ with $\chi_q: \Q_q^\times \to \C^\times$. When $p$ is odd, it 
has the following properties.
\begin{enumerate}
	\item
	For primes $q \neq p$, the character $\chi_q$ is 
	unramified and $\chi_q(q)=\Big( \frac{q}{p}\Big)$.
	\item 
	$\chi_p$ is ramified with conductor $p$ and the restriction
	$\chi|_{\Z_p^\times}$ factors through the unique quadratic
	character of $\mathbb{F}_p^\times$ with $\chi_p(p)=1$.
\end{enumerate}
By definition, $\chi_p$ is tamely ramified. Note that $\Q(\sqrt{\big( \frac{-1}{p} \big) p})/\Q$ is ramified only at $p$.  For $p=2$, the extensions $\Q(\sqrt{-1}),\Q(\sqrt{2})$, $\Q(\sqrt{-2})$ are ramified only at $2$, and the corresponding characters are denoted by $\chi_{-1}, \chi_2$ and $\chi_{-2}$ with conductors $2, 3, 3$ respectively.

\begin{theorem} \cite[Lemma 4.16]{deligne1} \label{epsilon factor while twisting}
	Let $\alpha, \beta \in \widehat{F^\times}$ be such that $a(\alpha) \geq 2a(\beta)$. If $\alpha(1+x)=\phi_{F}(cx)$ for $x\in\mathfrak{p}_{{F}}^r$ with $2r \geq a(\alpha)$, then the valuation of $c$ is $-(a(\alpha)+n(\phi_{F}))$ (if $a(\alpha)=0$ then $c= \mathfrak{p}_{{F}}^{-n(\phi_{F})}$) and $$\varepsilon(\alpha\beta, \phi_{F})= \beta^{-1}(c)\varepsilon(\alpha, \phi_{{F}}).$$ %where the valuation of $c$ is $-(a(\alpha)+n(\phi_{F}))$.
\end{theorem}

\section{Statement of results} \label{statement}
Let $f \in S_k(N, \epsilon)$ and the local representation $\pi_p$ be as above. We write $N=p^{N_p}N'$ with $p \nmid N'$ and $\epsilon=\epsilon_p \cdot \epsilon'$, where $\epsilon_p$ is the $p$-part and $\epsilon'$ is the prime-to-$p$ part of $\epsilon$. The conductor of $\epsilon_p$ is $p^{C_p}$ for some $C_p\leq N_p$. 
Let $\rho_{f,p}$ be the local representation of the Weil-Deligne group $W'(\Q_p)$ attached to $f$  at $p$. A complex $2$-dimensional Weil-Deligne representation of $W'(\Q_p)$ is a pair $(\phi, N)$ where 
\begin{enumerate} 
	\item 
	$\phi : W(\Q_p) \rightarrow {\rm GL}_2(\C)$ is a representation. 
	\item 
	$N$ is a nilpotent endomorphism of $\C^2$ such that $\phi(g)N\phi(g)^{-1}= \omega_1(g)N \quad \forall ~g \in W(\Q_p)$,
\end{enumerate} 
where $\omega_1$ is a character of $W(\Q_p)$ defined by $g: x \mapsto x^{\omega_1(g)} \,\, \forall~ x \in \overline{\mathbb{F}}_p$. A Weil-Deligne representation of $W'(\Q_p)$ for any vector space $\C^n$ is defined similarly. For more details, we refer to \cite{tate}. 
By the local Langlands correspondence, a $2$-dimensional Weil-Deligne representation $\rho_{f,p}=(\phi,N)$ corresponds uniquely to an irreducible admissible representation $\pi_p$ of ${\rm GL}_2(\Q_p)$ and vice-versa. The pair $(\phi, N)$ is said to the local parameter (or the $L$-parameter) of $\pi_p$. 

\begin{theorem}\label{q}
	Let $p$ be an odd prime and $q \neq p$. Then, $\varepsilon_q=\left(\frac{q}{p}\right)^{{\rm val}_q(a(\sym^2(\pi)))}$.
\end{theorem}

Let $\pi_p=\pi(\mu_1,\mu_2)$ be a principal series representation attached to $f$ at $p$. %where $\mu_1\mu_2^{-1}=|\cdot|^{\pm1}$. 
As $f$ is $p$-minimal, we have $\mu_1$ is unramified, $\mu_1(p)=\frac{a_p}{p^{(k-1)/2}}$, $\mu_1 \mu_2=\omega_p$ has conductor $p^{N_p}$ with $N_p \geq 1$ \cite[Prop. 2.8]{lw}. The $L$-parameter of $\pi_p$ is given by 
\begin{equation} \label{Lprin}
	\phi(x)= \begin{bmatrix}
		\mu_1(x) & \\ & \mu_2(x)	\end{bmatrix} , \, x\in W(\Q_p) \text{ and } N=0.
\end{equation}
In this case, the variation number is given by the theorem below.
\begin{theorem} \label{printhm}
	Let $f$ be a $p$-minimal newform with $\pi_{p}$ ramified principal series type. If $p$ is an odd prime with $N_p>1$, then $$\varepsilon_p=\begin{cases}
		p^{1-k}a_p^2, & \quad \text{if} \,\, p \equiv 1 \pmod{4}, \\
		ip^{1-k}a_p^2, & \quad \text{if} \,\, p \equiv 3 \pmod{4}.
	\end{cases}$$
	When $N_p=1$, the value of $\varepsilon_p$ is given in Table \ref{tab:table2}. 
	
	Let $p=2$. If $N_2\geq 4$, then $\varepsilon_2= i2^{1-2k}a_2^4\chi_{-1}(2)$. If $N_2=2$, or $N_2=3$ with $\omega_2|_{\mathbb{Z}_2^\times}= \chi_{-1}|_{\mathbb{Z}_2^\times}$, then $\varepsilon_2=\frac{i\omega_2^2(2)\chi_{-1}(2)}{2}$, otherwise  $\varepsilon_2=\frac{-\omega_2^4(2)\chi_{-1}(2)}{4} $.
\end{theorem}
If $\pi_p$ is a special representation, then $N_p=1$ and $\pi_{p} = \mu \otimes \mathrm{St}_2$ with $\mu$ unramified and $\mu(p)=\frac{a_p}{p^{(k-2)/2}}$ \cite[Prop. 2.8]{lw}. Its Langlands parameter is given by
\begin{equation}\label{Lspecial}
	\phi(x)= \begin{bmatrix}
		\mu(x)|x|^{\frac{1}{2}} & \\ & \mu(x)|x|^{-\frac{1}{2}}	\end{bmatrix} , \, x\in W(\Q_p) \text{ and } N=\begin{bmatrix}
		0 & 1 \\ 0 & 0
	\end{bmatrix}.
\end{equation}
\begin{theorem} \label{spthm}
	Let $f$ be a $p$-minimal newform with $\pi_{p}$ a special representation. For an odd prime $p$, we have
	\begin{equation*} 
		\varepsilon_p=\begin{cases}
			a_p^2p^{4-k}, & \text{ if } p \equiv 1 \pmod{4}\\
			-ia_p^2p^{4-k}, & \text{ if } p \equiv 3 \pmod{4}.\end{cases} 
	\end{equation*}
	If $p=2$, then $\varepsilon_2= -i2^{7-4k}a_p^8\chi_{-1}^3(2)$.
\end{theorem}
If the local representation is not of the above two types, then it is a supercuspidal representation. For an odd prime $p$, $\pi_p= {\rm Ind}_{W(K)}^{W(\Q_p)}(\varkappa)$ with $K/\Q_p$ quadratic. Its symmetric square transfer is either of Type I or II (cf. \S \ref{sup}).  In this case, the following theorem determines the variance number where $e$ denotes an element of $\mathcal{O}_K$ with valuation $-\frac{N_p}{2}$, see Equ. \ref{e}.
\begin{theorem} \label{supthm}
	Let $p$ be an odd prime and $f \in S_k(N, \epsilon)$ be a $p$-minimal newform with $\pi_{p}$ supercuspidal. For $K/\Q_p$ is unramified with $N_p>2$, $C_p\geq2$, we have
	$$\varepsilon_p = \begin{cases}
		\chi_p'(e), & \text{if} \,\, \sym^2(\pi_p) \text{ is of Type I},  \\
		1, & \text{if} \,\, \sym^2(\pi_p) \text{ is of Type II},
	\end{cases}$$
	where $\chi_p'=\chi_p \circ N_{K/\Q_p}$. If $K/\Q_p$ is a ramified extension with $C_p \geq 2$, then we have the following.
	\begin{enumerate}
		\item 
		For Type I representations, $\varepsilon_p =
		\begin{cases}
			1, & \text{if} \,\, (p,K|\Q_p)=1,  \\
			\Big( \frac{-1}{p} \Big), &  \text{if} \,\, (p,K|\Q_p)=-1.
		\end{cases}$ 
		\item 
		For Type II representations, $\varepsilon_p =1$.
	\end{enumerate}
\end{theorem}
The case $N_p>2$ with $C_p\leq 1$ is discussed in Theorem \ref{cp=0,1} whereas the situation $N_p=2$ is considered in \S \ref{np=2}. For $p=2$, Theorem \ref{supthm2} computes the variance number when the local representation is of dihedral supercuspidal type.

Using the above results, we determine the types of $\sym^2(\pi_p)$ in Corollaries \ref{maincoro} and \ref{mainkorop=2} in terms of the variation of global epsilon factors with respect to twisting by a quadratic character, where we also classify the quadratic extensions $K/\Q_p$ from which the local representation is induced (if it is of supercuspidal type). As this variation number is in terms of the prime-to-$p$ part of the conductor of $\sym^2(\pi)$, we determine the conductor of the symmetric square transfer of a newform in Theorem \ref{conductortheorem}.
	
\section{Local parameter of $\sym^2(\pi_p)$} \label{l-parameter}The symmetric square transfer of the standard representation of ${\rm GL}_2(\mathbb{C})$ is equivalent to the action of ${\rm GL}_2(\mathbb{C})$ on the vector space of homogeneous polynomials of degree $2$, see \S \ref{sym2}. Using its matrix representation, we calculate the $L$-parameter of $\sym^2$ of the local representation $\pi_p$.%at a prime $p$ attached to a modular form $f$.

\begin{proposition} \label{Lprinsym2} Let $f$ be a $p$-minimal newform such that $\pi_p$ is a principal series representation with $L$-parameter $(\phi, N)$ given in \eqref{Lprin}. Suppose the $L$-parameter of ${\rm sym}^2(\pi_p)$ is denoted by $({\rm sym}^2(\phi), N')$. Then,
\begin{equation} \label{prin}
	{\rm sym}^2(\phi)(x)=
	\begin{bmatrix}
		\mu_1^2(x) & &  \\ & \mu_1(x)\mu_2(x) &\\  & & \mu_2^2(x)
	\end{bmatrix} \, x\in W(\Q_p) \text{ and } N'=0.
\end{equation}
\end{proposition}
\begin{proof} Apply Equ. \ref{symsquare} to \eqref{Lprin} to get the desired result.  \end{proof}

%If $\pi_p$ is a special representation of ${\rm GL}_2(\Q_p)$ attached to a $p$-minimal newform $f$, then the Langlands parameter of its symmetric square is given below. 

\begin{proposition} \label{lspecial}
	Let $\pi_{p}$ be a special representation attached to a $p$-minimal form, i.e $\pi_{p} = \mu \otimes \mathrm{St}_2$ with $\mu$ unramified. Then $(\mu^2 \otimes {\rm St}_3, N')$ is the Langlands parameter of $\sym^2(\pi_{p})$ where $N'$ is the $3 \times 3$ matrix with the super diagonal entries $1$ and the rest entries $0$. 
\end{proposition}
\begin{proof}
	Note that the Langlands parameter of $\pi_p$ is given in \eqref{Lspecial}. Using Equ. \ref{symsquare} to \eqref{Lspecial}, the local parameter of $\sym^2(\phi)$ is of the form $\mu^2|\cdot| \oplus \mu^2 \oplus \mu^2|\cdot|^{-1}$ where $|\cdot|$ is the $p$-adic norm quasi-character.
	%{ \small 
%		\begin{eqnarray} \label{spsym2}
%			{\rm sym}^2(\phi)(w)= \begin{bmatrix}
%				\mu^2(w)|w| & &  \\ & \mu^2(w) & \\ & &  \mu^2(w)|w|^{-1}
%			\end{bmatrix}, w\in W(\Q_p).
%		\end{eqnarray}  
		%}
	To calculate $N'$, consider the commutative diagram:
	$\begin{tikzcd}
		{\rm GL}_2(\mathbb{C}) \arrow{r}{{\rm sym}^2}  & {\rm GL}_3(\mathbb{C}) \arrow{d}{d({\rm sym}^2)|_{t=0}}\\ {\mathfrak{gl}_2({\mathbb{C}})} \arrow{u}{\exp} \arrow{r} & {\mathfrak{gl}_3({\mathbb{C}})}
	\end{tikzcd}$
	where $\mathfrak{gl}_n(\mathbb{C})$ is the Lie algebra of ${\rm GL}_n(\mathbb{C})$. As $\exp(tN)= \begin{bmatrix}
		1 & t \\ 0 & 1
	\end{bmatrix}$ and ${\rm sym}^2(\exp(tN))=\begin{bmatrix}
		1 & t & t^2 \\ 0 & 1 & 2t \\ 0 & 0 & 1 
	\end{bmatrix}$, we have
	$A=d({\rm sym}^2(\exp(tN)))|_{t=0}= \begin{bmatrix}
		0 & 1 & 0 \\ 0 & 0 & 2 \\ 0 & 0 & 0
	\end{bmatrix}.$
	Taking $B= \begin{bmatrix}
		1 & 0 & 0 \\ 0 & 1 & 0 \\ 0 & 0 & 2
	\end{bmatrix}$, we have $BAB^{-1}=N'=\begin{bmatrix}
		0 & 1 & 0 \\ 0 & 0 & 1 \\ 0 & 0 & 0
	\end{bmatrix}.$
\end{proof}

Now, assume that $\pi_p$ is a supercuspidal representation. If $p$ is odd, then $\pi_p= {\rm Ind}_{W(K)}^{W(\Q_p)}(\varkappa)$ where $K/\Q_p$ is a quadratic extension with $\varkappa \neq \varkappa^\sigma, \sigma \in W(\Q_p)\backslash W(K)$ and $\varkappa^\sigma(x)=\varkappa(\sigma x \sigma^{-1}), x\in W(K)$. Let $(\phi, 0)$ be the $L$-parameter of $\pi_p$. After choosing a suitable basis, $\phi$ has following matrix form: 
\begin{equation} \label{lsup}
	\phi(x)= \begin{bmatrix}
		\varkappa(x) & \\ & \varkappa^\sigma(x)	\end{bmatrix} , \, x\in W(K) \text{ and } \phi(\sigma)= \begin{bmatrix}
		& 1 \\ \varkappa(\sigma^2) & 
	\end{bmatrix}.
\end{equation}
The conductor $a(\phi)$ of $\phi$ is given by 
\begin{equation} \label{indconductor} a(\phi)= {\rm dim}({\varkappa})~v(d_{K/\Q_p}) + f_{K/\Q_p}~ a(\varkappa), \end{equation} where $d_{K/\Q_p}$ is the discriminant of the field extension $K/\Q_p$ and $f_{K/\Q_p}$ is the residual degree, see \cite{serre}.  This gives us that 
\begin{equation} \label{np} N_p = a(\pi_p) =
\begin{cases}
	2a(\varkappa), \quad \quad \, \text{if } K/\Q_p \text{ is unramified},\\
	1+a(\varkappa), \quad \text{if }  K/\Q_p \text{ is ramified},
\end{cases}
\end{equation}
where $N_p$ is the highest power of $p$ such that $p^{N_p}$ divides the level of the modular form $f$. Now, applying the symmetric square transfer \eqref{symsquare} to \eqref{lsup}, we obtain that
\begin{equation} \label{supsym2}
	{\rm sym}^2(\phi)(x)= \begin{bmatrix}
		\varkappa^2(x) &  & \\ & \varkappa(x) \varkappa^\sigma(x) & \\ & & (\varkappa^\sigma)^2(x)	\end{bmatrix} , \, x\in W(K) \text{ and }  {\rm sym}^2(\phi)(\sigma)= \begin{bmatrix}
		&  & 1 \\ & \varkappa^\sigma(\sigma^2) & \\ (\varkappa^\sigma)^2(\sigma^2) & &
	\end{bmatrix}.
\end{equation}
In this case, we have the following description of the symmetric square transfer of $\pi_p$: $ \sym^2(\pi_p) = \Ind_{W(K)}^{W(\Q_p)} (\varkappa^2) \oplus \theta$, where $\theta$ is either $\omega_p$ or $\omega_p \omega_{K/\Q_p}$, see Prop. \ref{symndirectsum}.

\section{Non-supercuspidal  representations} \label{non-sup}
\noindent \textbf{Proof of Theorem \ref{q}}: Here we choose an additive character $\phi$ such that $n(\phi)=0$. As $\chi_q$ is unramified for $q \neq p$, using \cite[Equ. 5.5.1]{deligne1} we have that $\varepsilon_q = \frac{\varepsilon(\sym^2(\pi_{q}) \otimes \chi_q, \phi)}{\varepsilon(\sym^2(\pi_{q}),\phi)} = \chi_q(q^{a({\rm sym}^2(\pi_{q}))+n(\phi)\cdot{\rm dim}(\sym^2(\pi_{q}))})$. Now,  we obtain the desired value of $\varepsilon_q$ using the fact that $\chi_q(q)=\Big( \frac{q}{p}\Big)$. \qed 

\medskip

\noindent \textbf{Proof of Theorem \ref{printhm}}: We choose $\phi\in \widehat{\Q_p}$ with $n(\phi)=-1$. Using Equ. \ref{prin} and Prop. \ref{epsdef}, we have 
\begin{eqnarray} \label{prin1}
	\varepsilon_p= \frac{\varepsilon(\mu_1^2\chi_p, \phi) \varepsilon(\omega_p\chi_p, \phi) \varepsilon(\mu_1^{-2}\omega_p^2\chi_p,\phi) }{\varepsilon(\mu_1^2, \phi) \varepsilon(\omega_p, \phi) \varepsilon(\mu_1^{-2}\omega_p^2, \phi)}.
\end{eqnarray}
%where $\phi\in \widehat{\Q_p}$ with $n(\phi)=-1$.  
To calculate $\varepsilon_p$, we compute each local epsilon factors separately. First, consider the case where $p\geq3$ and $N_p >1$. As $\mu_1$ is unramified, by property $(\epsilon 1)$ of epsilon factors in \S \ref{epsilon} we have the following:
\begin{enumerate}
	\item 
	$\varepsilon(\mu_1^2\chi_p, \phi)= \mu_1^2(p)^{a(\chi_p)-1} \varepsilon(\chi_p, \phi)$ $= \varepsilon(\chi_p,\phi)$.
	
	\item 
	$\varepsilon(\mu_1^2, \phi)= \frac{1}{\mu_1^2(p)}$ as $\mu_1^2$ is unramified.
	
	\item 
	$\varepsilon(\mu_1^{-2}\omega_p^2\chi_p, \phi)= \mu_1^{-2}(p)^{a(\omega_p^2\chi_p)-1}\varepsilon(\omega_p^2\chi_p,\phi)= \mu_1^{-2}(p)^{a(\omega_p)-1}\varepsilon(\omega_p^2\chi_p,\phi)$. Here we use the fact that when $p\geq 3$ with $N_p >1$, we have $a(\omega_p^2)=a(\omega_p) > a(\chi_p)$ which gives $a(\omega_p^2\chi_p)=a(\omega_p)$. %as $a(\omega_p^2\chi_p)=a(\omega_p)$.
	
	\item $\varepsilon(\mu_1^{-2}\omega_p^2,\phi)=\mu_1^{-2}(p)^{a(\omega_p^2)-1}\varepsilon(\omega_p^2, \phi)=\mu_1^{-2}(p)^{a(\omega_p)-1}\varepsilon(\omega_p^2, \phi)$. %as $a(\omega_p^3)=a(\omega_p)$.
	
\end{enumerate} 

\noindent
Thus we deduce from Equ. \ref{prin1} that

\begin{eqnarray} \label{prin2}
	\varepsilon_p=\mu_1^2(p) \frac{\varepsilon(\chi_p, \phi) \varepsilon(\omega_p\chi_p,\phi) \varepsilon(\omega_p^2\chi_p,\phi)}{\varepsilon(\omega_p, \phi) \varepsilon(\omega_p^2, \phi)}.
\end{eqnarray}

%Note that $a(\omega_p^i)=a(\omega_p), i\in \{1, 2\}$, 
Now, taking $\chi=\omega_p^i , i\in \{1, 2\},$ in \cite[Lemma 2.4]{bm} %Lemma \ref{relation between multiplicative and additive character},
we have
$\omega_p^i(1+x)=\phi(c_ix) ~~\forall~ x\in \mathfrak{p}_{\Q_p}^r,~ 2r \geq N_p,$
for some $c_i\in \Q_p^\times$ with valuation $-(a(\omega_p^i)+n(\phi))$. As $a(\omega_p^2)=a(\omega_p)$, choosing $r$ suitably we can consider $c_1=c_2$ (say $c$). Now using Theorem \ref{epsilon factor while twisting}, we get %\begin{eqnarray} \label{prin3}
$\frac{\varepsilon(\omega_p\chi_p,\phi) \varepsilon(\omega_p^2\chi_p,\phi)}{\varepsilon(\omega_p,\phi) \varepsilon(\omega_p^2,\phi)}=\chi_p^{-1}(c^2)=1$ as $\chi_p$ is quadratic. Noting $\mu_1(p)=\frac{a_p}{p^{(k-1)/2}}$, from \eqref{prin2} we obtain %that $\varepsilon_p = \mu_1^2(p) \varepsilon(\chi_p,\phi)$. Now using \cite[Lemma 2.3]{bm}, we have the desired value of $\varepsilon_p$.

\begin{equation}
	\varepsilon_p %&=&  \mu_1^2(p)\varepsilon(\chi_p,\phi)\prod_{i=1}^2 \chi_p^{-1}(c_i)  \\
	= \mu_1^2(p) \varepsilon(\chi_p,\phi) \\
	\overset{\text{\cite[Lemma 2.3]{bm}}}{=} \begin{cases}
		p^{1-k}a_p^2, & \quad \text{if} \,\, p \equiv 1 \pmod{4}, \\
		ip^{1-k}a_p^2, & \quad \text{if} \,\, p \equiv 3 \pmod{4}.
	\end{cases}
\end{equation}
Let $N_p=1$ with $p\geq 3$. Since $N_p=C_p=a(\omega_p)$, from Equ. \ref{prin1}, we have,
\begin{equation}
	\begin{split}
		\varepsilon_p & = \mu_1(p)^{2-2a(\omega_p^2\chi_p)+2a(\omega_p^2)} \varepsilon(\chi_p,\phi) 
		\times  \prod_{j=1}^{2} \frac{\varepsilon(\omega_p^j\chi_p, \phi)}{\varepsilon(\omega_p^j,\phi)}
	\end{split}
\end{equation}
Let $t=2-2a(\omega_p^2\chi_p)+2a(\omega_p^2)$. The below table will be used to compute $\varepsilon(\omega_p^j, \phi)$ and $\varepsilon(\omega_p^j\chi_p, \phi), j \in \{1,2\}$.
	\begin{table}[h!]
	\begin{center}
		\begin{tabular}{|c|c|c|c|c|c|c|c|}
			\hline
			$o(\widetilde{\omega_p})$ & $a(\omega_p)$ & $a(\omega_p^2)$ & $a(\omega_p\chi_p)$ & $a(\omega_p^2\chi_p)$ & $t$\\
			\hline
			2 & 1 & 0 & 0 & 1 & 0 \\
			4 & 1 & 1 & 1 & 0 & 4 \\
			$> 4$ & 1 & 1 & 1 & 1 & 2 \\
			\hline
		\end{tabular}
		\vskip 1mm
		\caption{ }
		\label{tab:table1}
	\end{center}
\end{table}

Since both $\omega_p, \chi_p$ have conductor $1$, they can be thought of as characters of $\mathbb{F}_p^\times$. Let $\widehat{\mathbb{F}_p^\times}=\langle \chi_1 \rangle$ for some $\chi_1$. Set $\widetilde{\omega_p}= \omega_p^{-1}|_{\mathbb{Z}_p^\times}, \widetilde{\chi_p}=\chi_p|_{\mathbb{Z}_p^\times}$. Let $\circ(\widetilde{\omega_p})=m$. Write $\widetilde{\omega_p}= \chi_1^a$, and so $\widetilde{\omega_p}^j\widetilde{\chi_p}=\chi_1^{ja+\frac{p-1}{2}}$. Now by Equ. \ref{GKcoro} and the definition of epsilon factor, for $j=1, 2$, we have that
\begin{equation*} %\label{prin8}
	\varepsilon(\omega_p^j, \phi) = p^{-\frac{1}{2}} G_1(\chi_1^{ja}) =p^{-\frac{1}{2}}(-p)^{\frac{j}{m}}\Gamma_p \Big(\frac{j}{m} \Big),
\end{equation*}
Similarly, $\varepsilon(\omega_p^j\chi_p, \phi) = p^{-\frac{1}{2}}(-p)^{\frac{j}{m}+\frac{1}{2}}\Gamma_p \big(\frac{j}{m}+\frac{1}{2} \big)$.
%\begin{equation} \label{prin9}
%	\varepsilon(\omega_p^j\chi_p, \phi) = p^{-\frac{1}{2}} G_1(\chi_1^{ja+\frac{p-1}{2}}) =p^{-\frac{1}{2}}(-p)^{\frac{j}{m}+\frac{1}{2}}\Gamma_p \big(\frac{j}{m}+\frac{1}{2} \big).
%\end{equation}
Note that $\varepsilon(\alpha, \phi)=\frac{1}{\alpha(p)}$ if $\alpha$ is unramified and \cite[Lemma 2.3]{bm} computes it when $\alpha$ is quadratic tamely ramified. The values of $\varepsilon_p$ is given in Table \ref{tab:table2}.
%\begin{table}[h!]
%	\begin{center}
%		\begin{tabular}{|c|c|c|}
%			\hline
%			$\circ(\widetilde{\omega_p})$ & $\varepsilon_p$ & $c_p$ \\
%			\hline
%			\multirow{2}{*}{$2$} &
%			{ $p/2  \quad \text{if}\,\, p \equiv 1 \pmod{4}$} & \\ 
%			& $ip/2 \quad \text{if} \,\, p \equiv 3 \pmod{4}$ & \\
%			\hline
%			\multirow{2}{*}{4} &
%			\multirow{2}{*}{$ip^{\frac{5}{4}-2k}a_p^4\frac{\Gamma_p(\frac{3}{4})}{\Gamma_p(\frac{1}{2})}$} & 
%			\multirow{2}{*}{} \\ & {} & {} \\
%			\hline
%			\multirow{2}{*}{$>4$} &
%			\multirow{2}{*}{$\frac{p^{3-2k}a_p^4}{4\omega_p^2(p)}$} & 
%			\multirow{2}{*}{} \\ & {} & {}\\
%			\hline
%		\end{tabular}
%		\vskip 1mm
%		\caption{ }
%		\label{tab:table2}
%	\end{center}
%\end{table}	

\begin{table}[h!]
	\begin{center}
		\begin{tabular}{|c|c|c|c|}
			\hline
			$\circ(\widetilde{\omega_p})$ & $\varepsilon_p$ & $\circ(\widetilde{\omega_p})$ & $\varepsilon_p$  \\
			\hline
			\multirow{2}{*}{$2$} &
			{$p/2  \quad \text{if}\,\, p \equiv 1 \pmod{4}$} &
			\multirow{2}{*}{$>4$} &
			{$-p^{2-k}a_p^2 b_p  \quad \text{if}\,\, p \equiv 1 \pmod{4}, b_p= \frac{\prod_{j=1}^{2}\Gamma_p(\frac{j}{m}+\frac{1}{2})}{\prod_{j=1}^{2}\Gamma_p(\frac{j}{m})}$} \\ 
			& {$ip/2 \quad \text{if} \,\, p \equiv 3 \pmod{4}$} & { } & {$-ip^{2-k}a_p^2 b_p  \quad \text{if}\,\, p \equiv 3 \pmod{4}$} \\
			\hline
			\multirow{2}{*}{$4$} &
			\multirow{2}{*}{$ip^{\frac{5}{4}-2k}a_p^4\frac{\Gamma_p(\frac{3}{4})}{\Gamma_p(\frac{1}{2})}$} &
			\multirow{2}{*}{} &
			\multirow{2}{*}{} \\ & {} & {} & {} \\
			\hline
		\end{tabular}
		\vskip 1mm
		\caption{ }
		\label{tab:table2}
	\end{center}
\end{table}	
Let $p=2$ and $N_p\geq 4$. To calculate $\varepsilon_2$, we consider the quadratic character $\chi_p=\chi_{-1}$ (see \S \ref{epsilon}). Using the same argument as in the odd prime case, we get that
\begin{eqnarray} \label{prineve}
	\varepsilon_p=\mu_1^4(p) \frac{\varepsilon(\chi_p, \phi) \varepsilon(\omega_p\chi_p,\phi) \varepsilon(\omega_p^2\chi_p,\phi)}{\varepsilon(\omega_p, \phi) \varepsilon(\omega_p^2, \phi)}.
\end{eqnarray}
Similar calculation after Equ. \ref{prin2} and \cite[Lemma 4.2]{bmm} gives $\varepsilon_2= i2^{1-2k}a_2^4\chi_{-1}(2)$. Next, assume $N_2=2$. As $\circ(\widetilde{\omega_2})=2$, we can easily calculate $\varepsilon_2$ and get the desired value. % $\varepsilon_2=\frac{i\omega_2^2(2)\chi_{-1}(2)}{2}$. 
Let $N_2=3$. If $\omega_2|_{\mathbb{Z}_2^\times}= \chi_{-1}|_{\mathbb{Z}_2^\times}$ then $\omega_2\chi_{-1}$ is unramified,  and we obtain $\varepsilon_2=\frac{i\omega_2^2(2)\chi_{-1}(2)}{2}$. Otherwise, $a(\omega_2\chi_{-1})=1$ and we have $\varepsilon_2= \frac{-\omega_4^4(2)\chi_{-1}(2)}{4}$. The case $N_2=1$ cannot happen.
 \qed

\medskip

Before proving the theorem for a special representation, let us discuss how to compute its epsilon factor. If $\pi_p$ is a special representation, then the $L$-parameter of its symmetric square transfer is $(\mu^2 \otimes {\rm St}_3, N')$, see Prop. \ref{lspecial}. 
Set $\rho:=\mu^2 \otimes \rm{St}_3$. Using Prop. \ref{epsdef} we now compute the $\varepsilon$-factor of $\rho'=(\rho, N')$. 
\begin{proposition}  \label{specialprop}
	We have
	\begin{eqnarray*} 
		\varepsilon(s, \rho', \phi)=
		\begin{cases}
			\varepsilon(s, \rho, \phi), & \quad \text{if} \,\, \mu \text{ is ramified}, \\
			\mu^4(p)p^{-2s-1} \varepsilon(s, \rho, \phi), & \quad \text{if} \,\, \mu \text{ is unramified}.
		\end{cases}
	\end{eqnarray*}
\end{proposition}
\begin{proof}
	Let $W$ denote the space of $\mu^2$. So $V=W \otimes \C^3$ is the space of $\rho'$. If $\{e_0, e_1, e_2\}$ is the standard basis of $\C^3$, then $\rm{ker} (N')=\C e_0$. This implies that 
	%\begin{eqnarray} \label{spaces}
		$V^I=W^I \otimes \C^3 \text{ and } V_{N'}^I=W^I \otimes \C e_0.$
	%\end{eqnarray}
	
	Let $\mu$ be ramified. %Thus, $W^I=\{0\}$ and so $V^I=\{0\}$. 
	Thus, $W^I=V^I=\{0\}$. So Prop. \ref{epsdef} implies that $\varepsilon(s, \rho', \phi)=\varepsilon(s, \rho, \phi).$
	%\begin{eqnarray*} \varepsilon(s, \rho', \phi)=\varepsilon(s, \rho, \phi).\end{eqnarray*}
	Now, suppose that $\mu$ is unramified. From the proof of Prop. \ref{lspecial}, 
	$\rho= \displaystyle{\bigoplus_{i+j=2, 0 \leq j \leq 2} \mu^2 |.|^{\frac{i-j}{2}} }$  where the characters are unramified. Note that $W=W^I$ is $1$-dimensional. Thus, $V^I/V^I_{N'} = \oplus_{j \ne 0} (W\otimes e_j)$ with $-\rho(\Phi)$ acting by $-\mu^2(p)|p|^{\frac{i-j}{2}}$ such that $i+j=2$. Therefore, $\det (-\rho(\Phi) p^{-s} | V^I/V^I_{N'})=\mu^4(p)p^{-2s}|p|$. So by Prop. \ref{epsdef} we obtain that
	$
		\varepsilon(s, \rho', \phi)=\mu^4(p)p^{-2s-1} \varepsilon(s, \rho, \phi),
	$
	proving the proposition.
\end{proof}

We are now in a position to prove Theorem \ref{spthm}.
%\smallskip 

\noindent \textbf{Proof of Theorem \ref{spthm}}: Note that $\varepsilon(\frac{1}{2}, \rho, \phi)=\varepsilon(\rho, \phi)$, explained in \cite[\S 3]{bm}. % The $L$- parameter of ${\rm \sym}^3(\pi_{f, p})$ is $\rho'=(\rho, N')$ where $\rho=\mu^2 \otimes {\rm St}_3$ and $N'$ is given in Prop. \ref{lspecial}. 
Recall, $\rho$ is a direct sum of the unramified characters $\mu^2 |.|^{\frac{i-j}{2}}$ for $0 \leq j \leq 2$ such that $i+j=2$. By the properties of epsilon factors in \S \ref{epsilon}, $\varepsilon(\rho, \phi)={\displaystyle \prod_{0\leq j \leq 2, i+j=2}} \varepsilon(\mu^2 |.|^{\frac{i-j}{2}}, \phi)$ whereas $\varepsilon(\rho\otimes \chi_p, \phi)={\displaystyle \prod_{0\leq j \leq 2, i+j=2}} \varepsilon(\mu^2 |.|^{\frac{i-j}{2}}\chi_p, \phi).$ Now,
\begin{enumerate}
	\item 
	$\varepsilon(\mu^2 |.|^{\frac{i-j}{2}}, \phi)=(\mu^2 |.|^{\frac{i-j}{2}})(\frac{1}{p})=\frac{p^{\frac{i-j}{2}}}{\mu^2(p)}$, by the property of $\varepsilon$-factors of unramified characters.
	\item 
	$\varepsilon(\mu^2 |.|^{\frac{i-j}{2}} \chi_p, \phi)=\varepsilon(\chi_p, \phi)$, by $(\epsilon 1)$.
\end{enumerate}
Therefore, using Prop. \ref{specialprop} and $\mu(p)=\frac{a_p}{p^{(k-2)/2}}$ it follows that
\begin{eqnarray*}
	\varepsilon_p &=&\frac{\varepsilon(\rho' \otimes \chi_p, \phi)}{\varepsilon(\rho', \phi)} =\frac{1}{\mu^4(p) p^{-2 \cdot \frac{1}{2}-1}} ~~  {\displaystyle \prod_{0\leq j \leq 2, i+j=2}} ~~ \frac{\varepsilon(\mu^2 |.|^{\frac{i-j}{2}} \chi_p, \phi)}{\varepsilon(\mu^2 |.|^{\frac{i-j}{2}}, \phi) }  \\
	&=& \frac{p^2}{\mu^{4}(p)} ~~  {\displaystyle \prod_{0\leq j \leq 2, i+j=2}}  ~~ \frac{\mu^2(p) \varepsilon(\chi_p, \phi)}{p^{\frac{i-j}{2}}}   \\
	&=& \frac{p^2}{\mu^{4}(p)}  \times \mu^{6}(p) \varepsilon(\chi_p, \phi)^3 \\ %\text { by \cite[Lemma 2.3]{bm}}
	&=& \begin{cases}
		a_p^2p^{4-k} & \text{ if } p \equiv 1 \pmod{4}\\
		-ia_p^2p^{4-k} & \text{ if } p \equiv 3 \pmod{4} \end{cases}  \text { by \cite[Lemma 2.3]{bm}} 
\end{eqnarray*}
Let $p=2$. As the twisting character $\chi_p=\chi_{-1}$ has conductor $2$, we get $\varepsilon(\mu^2 |.|^{\frac{i-j}{2}} \chi_p, \phi)=\mu^2(p) |p|^{\frac{i-j}{2}} \varepsilon(\chi_p, \phi)$ $\overset{\text{\cite[Lemma 4.2]{bmm}}}= \mu^2(p) |p|^{\frac{i-j}{2}}  \cdot \frac{\sqrt{-1}\chi_{-1}(2)}{2}$. Then, in a similar way as above, $\varepsilon_2= -i2^{7-4k}a_p^8\chi_{-1}^3(2)$.  \qed

\section{Supercupidal Representation} \label{sup}
\subsection{The case $p$ odd} For an odd prime $p$, let $\pi_p$ be a supercuspidal representation i.e., $\pi_p= {\rm Ind}_{W(K)}^{W(\Q_p)}(\varkappa)$ where $K/\Q_p$ is a quadratic extension with $\varkappa \neq \varkappa^\sigma, \sigma \in W(\Q_p)\backslash W(K)$ and $\varkappa^\sigma(x)=\varkappa(\sigma x \sigma^{-1}), x\in W(K)$. The central character of $\pi_p$ is $\varkappa|_{\Q_p^\times} \cdot \omega_{K/\Q_p}$, where $\omega_{K/\Q_p}$ is the quadratic character of $\Q_p^\times$ associated to $K/\Q_p$ such that $\omega_{K/\Q_p} ({N_{K/\Q_p}(K^\times)})$ $= 1$. Therefore,  $\varkappa|_{\Q_p^\times} \cdot \omega_{K/\Q_p}=\omega_p$ where $\omega_p$ is the $p$-th component of the adelization of $\epsilon$. Now, evaluating at $N_{K/\Q_p}(x)$ for $x \in K^\times$, we get 
\begin{equation} \label{centralcharacter}
	\varkappa^\sigma=\varkappa^{-1} \omega_p' \text{ on } K^\times \text{ where } \omega_p'=\omega_p \circ N_{K/\Q_p}.
\end{equation}
Since $\omega_p= \epsilon_p^{-1}$ on $\Z_p^\times$ \cite[\S 2]{lw}, we have 
\begin{equation} \label{centralcharacter1}
	\varkappa^\sigma=\varkappa^{-1} \epsilon_p' \text{ on } \mathcal{O}_K^\times \text{ where } \epsilon_p'=\epsilon_p^{-1} \circ N_{K/\Q_p}.
\end{equation}

\subsubsection{Properties of $\sym^2(\pi_p)$} Let us discuss a few properties of symmetric square of $\pi_p$.
\begin{proposition}\label{symndirectsum}
	If $\pi_p= {\rm Ind}_{W(K)}^{W(\Q_p)}(\varkappa)$ is a supercuspidal representation, then $\sym^2(\pi_p)$ is given by \begin{equation} \label{symsup} \sym^2(\pi_p) = \Ind_{W(K)}^{W(\Q_p)} (\varkappa^2) \oplus \theta \end{equation}
	where $\theta$ is either $\omega_p$ or $\omega_p \omega_{K/\Q_p}$.
\end{proposition}

\begin{proof}
	As $\varkappa^\sigma(\sigma^2)=\varkappa(\sigma^2)$, from Equ. \ref{supsym2} we obtain that ${\rm sym}^2(\pi_p) = {\rm Ind}_{W(K)}^{W(\Q_p)} (\varkappa^2) \oplus \theta$. Here $\theta : W(\Q_p)\to \mathbb{C}^\times$ is given by $\theta(x)=(\varkappa \varkappa^{\sigma})(x), x\in W(K)$ and $\theta(\sigma)=(\varkappa^\sigma)(\sigma^2)$.
	
	Note that $(\varkappa \varkappa^\sigma)^\sigma=\varkappa\varkappa^\sigma$. Hence, the representation ${\rm Ind}_{W(K)}^{W(\Q_p)} (\varkappa \varkappa^\sigma)$ is reducible and we have  \begin{equation} \label{eq1} {\rm Ind}_{W(K)}^{W(\Q_p)} (\varkappa \varkappa^\sigma) =\psi \oplus \psi \omega_{K/\Q_p} \end{equation}  where $\varkappa \varkappa^\sigma= \psi \circ N_{K/\Q_p}$ for some character $\psi$ of $\Q_p^\times$. Now, using the Frobenius reciprocity, we have 
	\[{\mathrm {Hom}}_{W(\Q_p)}\big(\theta, {\rm Ind}_{W(K)}^{W(\Q_p)}(\varkappa \varkappa^\sigma)\big) \simeq {\mathrm {Hom}}_{W(K)}\big(\theta|_{W(K)}, \varkappa \varkappa^\sigma \big).\] Hence from Equ. \ref{eq1}, we get that $\theta$ is either $\psi$ or $\psi \omega_{K/\Q_p}$. The Equ. \ref{centralcharacter} now gives us the desired $\theta$. %$\theta$ is a character of $\Q_p^\times$ such that $\varkappa \varkappa^{\sigma} =\theta \circ N_{K/\Q_p}$. This can be seen as follows: using the fact $\pi_p$ is irreducible if and only if $\varkappa \neq \varkappa^\sigma$ on $K^\times$, we deduce ${\rm Ind}_{W(K)}^{W(\Q_p)} (\varkappa \varkappa^{\sigma})$ is reducible. Hence, the character $\varkappa \varkappa^{\sigma}$ factors through the norm map $N_{K/\Q_p}$, and we write $\varkappa \varkappa^{\sigma}=\theta \circ N_{K/\Q_p}$. Now using the Equ. \ref{cite} we complete the proof of the proposition.
\end{proof}
Note that the representation ${\rm Ind}_{W(K)}^{W(\Q_p)} (\varkappa^2)$ may not be always irreducible. Depending upon its irreducibility, $\sym^2(\pi_p)$ can be classified as follows:
\begin{itemize}
	\item \textbf{(Type I)} ${\rm Ind}_{W(K)}^{W(\Q_p)} (\varkappa^2)$ is irreducible.
	\item \textbf{(Type II)} ${\rm Ind}_{W(K)}^{W(\Q_p)} (\varkappa^2)$ is reducible.
\end{itemize}
For a Type II representation, we have $\varkappa^2 = \varphi \circ N_{K/\Q_p}$ for some character $\varphi$ of $\Q_p^\times$. In this case, ${\rm Ind}_{W(K)}^{W(\Q_p)} (\varkappa^2)=\varphi \oplus \varphi \omega_{K/\Q_p}$. % where $\omega_{K/\Q_p}$ is the unique quadratic character of $\Q_p^\times$ such that it is trivial on the norm group $N_{K/\Q_p}(K^\times)$. 
Let us now recall the following useful lemma \cite[Lemma $1.8$]{tunnell} to compute the conductor of a character obtained by composing with the norm map.
\begin{lemma}\label{conductor norm map}
	Let $K/F$ be a quadratic separable extension. If $\chi \in \widehat{F^\times}$, then $f a(\chi\circ N_{K/F})=a(\chi)+a(\chi\omega_{K/F})-a(\omega_{K/F})$, where $f$ is the residual degree. For a non-trivial $\psi \in \widehat{K}$, 
	$n(\psi \circ \mathrm{Tr}_{K/F})=(2/f)n(\psi)+v(d(K/F))$.
	Here, %$\omega_{E|K}$ denotes the non-trivial character of $K^\times$ with $\omega_{E/K}(N_{E/K} E^\times)=1$ kernel $N_{E/K} (E^\times)$ %equal to the group of norms from $E^\times$ to$K^\times$ and 
	$v(d(K/F))$ is the valuation of the discriminant of $K/F$.
\end{lemma}
\begin{proposition}\label{types}
	Type II representations cannot occur in the following cases.
	\begin{enumerate}
		\item 
		$C_p \leq 1, N_p \geq 4$ with $K/\Q_p$ unramified.
		\item 
		$C_p \leq 1, N_p \geq 3$ or $C_p=2, N_p=3$ together with $K/\Q_p$ ramified.
	\end{enumerate}
\end{proposition}
\begin{proof}
	As ${\rm Ind}_{W(K)}^{W(\Q_p)} (\varkappa^2)$ is reducible, we have $\varkappa^2=(\varkappa^2)^\sigma$. By \eqref{centralcharacter1}, we deduce that $\varkappa^4=\epsilon_p'^2$ on $\mco_{K}^\times$. % where $\epsilon_p'=\epsilon_p^{-1} \circ N_{K/\Q_p}$. 
	
	First consider the case (1). By Lemma \ref{conductor norm map}, $\epsilon_p'$ is either unramified or tamely ramified depending upon $C_p=0$ or $1$ respectively. In both cases, we deduce that $a(\varkappa^4) \leq 1$ which implies that $a(\varkappa) \leq 1$ as $p$ is assumed to be odd. Now, since $N_p=2a(\varkappa) \geq 4$ by assumption we arrive at a contradiction.
	Next, assume the case (2). If $C_p \leq 1, N_p \geq 3$ then again by Lemma \ref{conductor norm map}, $a(\epsilon_p') \leq 1$. Hence, $a(\varkappa^4) \leq 1$ and the proof follows using the same argument as before. The case $C_p=2, N_p=3$ follows the same way.
\end{proof}

\subsubsection{Computation of $\varepsilon_p$} For Type I and II representations, using the property $(\epsilon 2)$, the variance numbers are respectively given by
\begin{equation} \label{supvariance}
	\varepsilon_p= \frac{\varepsilon(\varkappa^2 \chi_p', \phi_K) \varepsilon(\theta \chi_p, \phi)}{\varepsilon(\varkappa^2, \phi_K) \varepsilon(\theta, \phi)} \quad \text{and} \quad \frac{\varepsilon(\varphi \chi_p, \phi) \varepsilon(\varphi \omega_{K/\Q_p} \chi_p, \phi) \varepsilon(\theta \chi_p, \phi)}{\varepsilon(\varphi, \phi) \varepsilon(\varphi \omega_{K/\Q_p}, \phi) \varepsilon(\theta, \phi)},
\end{equation}
where $\chi_p'=\chi_p \circ N_{K/\Q_p}$ and $\phi_K=\phi \circ \text{Tr}_{K/\Q_p}$.

\medskip
From now on we consider an additive character of conductor zero unless otherwise specified. By \cite[Lemma 2.4]{bm}, for an additive character $\phi$ of $\Q_p$ with $n(\phi)=0$, we have \begin{equation} \label{eps}\theta(1+x)=\phi(cx) \text{ for } x \in \mathfrak{p}^r \text{ with } 2r\geq C_p \text{ and } v_p(c)=-C_p. \end{equation} 
\begin{lemma} \label{eps=1}
	Let $C_p \geq 2$. We have %\begin{equation} \label{epsratio}
		$\frac{\varepsilon(\theta \chi_p, \phi)}{\varepsilon(\theta, \phi)} =1$.% \end{equation}
\end{lemma}

\begin{proof} Since $\theta$ is either $\omega_p$ or $\omega_p \omega_{K/\Q_p}$, and $a(\omega_p)=C_p$, we have $a(\theta)\geq 2$. Also, note that the element $c$ in the above Equation \eqref{eps} has the form $p^nu$ for some $u \in \Z_p^\times$. Now, replacing $\phi$ by $\phi_u$, we can take $c$ to be $p^n$. Using Theorem \ref{epsilon factor while twisting}, we obtain that $\frac{\varepsilon(\theta \chi_p, \phi)}{\varepsilon(\theta, \phi)} = \chi_p(p^n)^{-1}=1.$ Hence the proof follows. \end{proof}

\begin{lemma} \label{eps=reducible}
	The ratio $$ \label{epsratio=reducible} \frac{\varepsilon(\varphi \chi_p, \phi) \varepsilon(\varphi \omega_{K/\Q_p} \chi_p, \phi)}{\varepsilon(\varphi, \phi) \varepsilon(\varphi \omega_{K/\Q_p}, \phi)} = 1.$$
\end{lemma}

\begin{proof} For the above chosen additive character in \eqref{eps}, by \cite[Lemma 2.4]{bm} there exists an element $d \in \Z_p^\times$ such that 
%\begin{equation} \label{d1}
	$\varphi(1+x)=\phi(dx) \text{ for } x \in \mathfrak{p}^r \text{ with } 2r\geq a(\varphi).$
%\end{equation} 
As $\omega_{K/\Q_p}$ is either unramified or tamely ramified depending upon $K/\Q_p$ is unramified or ramified, we deduce that this equation is also valid when we replace $\varphi$ by $\varphi \omega_{K/\Q_p}$. Therefore, Theorem \ref{epsilon factor while twisting} implies that the desired ratio equals to $\chi_p(d^2)$, which is $1$. This completes the proof of the lemma. %$\frac{\varepsilon(\varphi \chi_p, \phi) \varepsilon(\varphi \omega_{K/\Q_p} \chi_p, \phi)}{\varepsilon(\varphi, \phi) \varepsilon(\varphi \omega_{K/\Q_p}, \phi)}=\chi_p(d^2)=1$. 
\end{proof}
Let us now prove Theorem \ref{supthm}.

%\medskip 

\noindent\textbf{Proof of Theorem \ref{supthm}}: Let $K/\Q_p$ be unramified with $N_p>2$, $C_p\geq2$. First assume that $\sym^2(\pi_p)$ is of Type I. By \cite[Lemma 2.4]{bm}, there exists an element $e \in \mco_{K}$ of valuation $-(a(\varkappa^2)+n(\phi_K))=-a(\varkappa)$ [by Lemma \ref{conductor norm map}] such that 
\begin{equation} \label{e} \varkappa^2(1+x)=\phi_K(ex) \text{ for } x \in \mathfrak{p}_K^r \text{ with } 2r>a(\varkappa^2). \end{equation} 
Hence, by Theorem \ref{epsilon factor while twisting}, the ratio $\frac{\varepsilon(\varkappa^2 \chi_p', \phi_K)}{\varepsilon(\varkappa^2, \phi_K)} = \chi_p'(e)$. Now using Lemma \ref{eps=1} in \eqref{supvariance}, we have $\varepsilon_p=\chi_p'(e)$.
If $\sym^2(\pi_p)$ is of Type II, then combining Lemmas \ref{eps=1} and \ref{eps=reducible}, we obtain from Equ. \ref{supvariance} that $\varepsilon_p=1$.

Next, assume that $K/\Q_p$ is ramified. There are two choices for $K$. Note that $K=\Q_p(\sqrt{-p})$ or $\Q_p(\sqrt{-p\zeta_{p-1}})$ depending upon $p$ is a norm of some element of $K^\times$ or not. We can choose a uniformizer $\pi=\sqrt{-p}$ or $\sqrt{-p \zeta_{p-1}}$ depending upon $K$, and write $K=\Q_p(\pi)$. We have $a(\chi_p')=0$ and $n(\phi_K)=1$ (using Lemma \ref{conductor norm map}).
\begin{enumerate}  [wide, labelwidth=!, labelindent=0pt]
	\item
	Let $\sym^2(\pi_p)$ be of Type I. By the property of epsilon factors for unramified twists we have that $\frac{\varepsilon(\varkappa^2 \chi_p', \phi_K)}{\varepsilon(\varkappa^2, \phi_K)} = \chi_p'(\pi)^{a(\varkappa^2)+1}$. Therefore, from \eqref{supvariance} and Lemma \ref{eps=1}, $\varepsilon_p=\chi_p'(\pi)^{a(\varkappa^2)+1}$. As $f$ is assumed to be $p$-minimal, we have $a(\varkappa)$ is even by \cite[Proposition 5.4]{bm}, making $a(\varkappa^2)$ even. Thus we deduce that 
	\begin{eqnarray} \label{variation}
		\varepsilon_p=\chi_p(N_{K|\Q_p}(\pi))=\Big( \frac{N(\pi)/p}{p} \Big).
	\end{eqnarray}
	Since $N_{K|\Q_p}(\pi)=-\pi^2$, we obtain that
	$
	\varepsilon_p=\Big( \frac{-\pi^2/p}{p} \Big).
	$
	Therefore when $p$ is odd, we deduce that:
	\begin{equation} \label{residuesymbol}
		\varepsilon_p =
		\begin{cases}
			\Big(\frac{1}{p}\Big)=1, 
			\quad \text{if} \,\, (p,K|\Q_p)=1,  \\
			\Big(\frac{\zeta_{p-1}}{p}\Big)=\Big( \frac{-1}{p} \Big),
			\quad \text{if} \,\, (p,K|\Q_p)=-1.
		\end{cases}
	\end{equation}
	
	\medskip
	\item 
	If $\sym^2(\pi_p)$ is of Type II, then Lemmas \ref{eps=1} and \ref{eps=reducible} together with Equ. \ref{supvariance} implies that $\varepsilon_p=1$.
\end{enumerate}\qed

It remains to consider the case where $C_p \leq 1$. In this case, we choose an additive character $\phi$ of conductor $-1$. 
Let $A_\theta= \frac{\varepsilon(\theta \chi_p, \phi)}{\varepsilon(\theta, \phi)}$ where $\theta$ is either $\omega_p$ or $\omega_p\omega_{K/\Q_p}$. We calculate this quantity in the next two lemmas below depending upon $C_p=0$ or $1$.
\begin{lemma} \label{cp0}
	For $C_p=0$ with $N_p>2$, the quantity $A_\theta$ is given as follows. \begin{enumerate}
		\item 
		If $\theta=\omega_p$ or $\omega_p\omega_{K/\Q_p}$ with $K/\Q_p$ unramified, then
		\[ 
		A_\theta = \begin{cases}
			\theta(p) & \text{ if } p \equiv 1 \pmod{4}\\
			i\theta(p) & \text{ if } p \equiv 3 \pmod{4}.\end{cases}.
		\] 
		\item 
		If $\theta=\omega_p\omega_{K/\Q_p}$ with $K/\Q_p$ ramified, then we have
		\[ 
		A_\theta = \begin{cases}
			\theta^{-1}(p) &  \text{ if } p \equiv 1 \pmod{4}\\
			-i\theta^{-1}(p) & \text{ if } p \equiv 3 \pmod{4}.\end{cases}.
		\]
	\end{enumerate}
\end{lemma}
\begin{proof}
	Recall that $\omega_{K/\Q_p}$ is either unramified or tamely ramified depending upon $K/\Q_p$ is unramified or ramified. For the first part, as $C_p=0$ both characters $\omega_p$ and $\omega_p\omega_{K/\Q_p}$ are unramified. %and we get the desired result using \cite[Lemma 3.2]{bmm}. 
	Using the property of epsilon factors of an unramified character and ($\epsilon 1$) in \S \ref{epsilon}, we deduce that $A_\theta=\theta(p) \varepsilon(\chi_p, \phi)$. Now we get the result by using \cite[Lemma 4.2]{bmm}.
	For the second part, $a(\omega_{K/\Q_p})=1$. Note that $\theta=\omega_p\omega_{K/\Q_p}$ is a tamely ramified quadratic character, so we use the same lemma of {\it loc. cit.} to calculate its epsilon factor. On the other hand, the character $\omega_p\omega_{K/\Q_p}\chi_p$ is unramified, so its epsilon factor is $\frac{1}{(\omega_p\omega_{K/\Q_p})(p)}$ as $\chi_p(p)=1$. This completes the proof. 
\end{proof}

\begin{lemma} \label{cp1}
	Let $N_p>2$ and $C_p=1$. The quantity $A_\theta$ is given as follows. \begin{enumerate} 
		\item 
		Let $\theta=\omega_p$ or $\omega_p\omega_{K/\Q_p}$ with $K/\Q_p$ unramified. If $\omega_p^2|_{\Z_p^\times} \neq 1$, then $A_\theta=(-p)^\frac{1}{2} \Big \{\Gamma_p\Big(\frac{1}{n}+\frac{1}{2}\Big)/\Gamma_p\Big(\frac{1}{n}\Big) \Big \}$ where $n=\circ(\widetilde{\omega_p})$; otherwise 
		\[ 
		A_\theta = \begin{cases}
			\theta^{-1}(p) &  \text{ if } p \equiv 1 \pmod{4}\\
			-i\theta^{-1}(p) & \text{ if } p \equiv 3 \pmod{4}.\end{cases}.
		\]
		\item 
		Let $\theta=\omega_p\omega_{K/\Q_p}$ with $K/\Q_p$ ramified. If $\omega_p^2|_{\Z_p^\times} \neq 1$, then $A_\theta=(-p)^\frac{1}{2} \Big \{\Gamma_p\Big(\frac{1}{n}\Big)/\Gamma_p\Big(\frac{1}{n}+\frac{1}{2}\Big) \Big \}$ where $n=\circ(\widetilde{\omega_p})$; otherwise 
		\[ 
		A_\theta = \begin{cases}
			\theta(p) & \text{ if } p \equiv 1 \pmod{4}\\
			i\theta(p) & \text{ if } p \equiv 3 \pmod{4}.\end{cases}.
		\] 
	\end{enumerate}
\end{lemma}

\begin{proof}
	%We use the method similar to \cite[Lemma 4.10]{bmm} to get the desired result.
	\begin{enumerate} %[wide, labelwidth=!, labelindent=0pt]
		\smallskip 
		\item 
		Let $\theta=\omega_p$. Note that $a(\omega_p)=1$. Suppose that $\omega_p|_{\Z_p^\times}$ is quadratic. Then the character $\omega_p \chi_p$ is unramified, so its epsilon factor equals to $\omega_p^{-1}(p)$. Now we use \cite[Lemma 4.2]{bmm} to determine $\varepsilon(\omega_p, \phi)$ and we get the desired result. 
		
		Next, suppose that $\omega_p^2|_{\Z_p^\times} \neq 1$. In this case, $a(\omega_p \chi_p)=1$. Let $\circ(\widetilde{\omega_p})=n$. Write $\widetilde{\omega_p}=\chi_1^a$ where $\widehat{\mathbb{F}_p^\times} = \langle \chi_1 \rangle$. Then, $\widetilde{\omega_p} \widetilde{\chi_p}=\chi_1^{a+\frac{p-1}{2}}$. Now using Equ. \ref{GKcoro} and the definition of epsilon factors, 
		$\varepsilon(\omega_p, \phi)=p^{-1/2}G_1(\chi_1^a)=p^{-1/2}(-p)^{\frac{a}{p-1}} \Gamma_p(\frac{a}{p-1})=p^{-1/2}(-p)^{\frac{1}{n}} \Gamma_p(\frac{1}{n})$. Similarly, $\varepsilon(\omega_p \chi_p, \phi)=p^{-1/2}(-p)^{\frac{1}{n}+\frac{1}{2}} \Gamma_p(\frac{1}{n}+\frac{1}{2})$. This completes the proof for $\theta=\omega_p$. By the same argument, we get the value of $A_\theta$ when $\theta=\omega_p\omega_{K/\Q_p}$ with $K/\Q_p$ unramified. 
		%		Therefore,
		%		\begin{eqnarray*}
			%			A_\theta=(-p)^\frac{1}{2} \Big \{\Gamma_p\Big(\frac{1}{n}+\frac{1}{2}\Big)/\Gamma_p\Big(\frac{1}{n}\Big) \Big \}.
			%		\end{eqnarray*}
		\smallskip 
		\item	
		Let $\theta=\omega_p\omega_{K/\Q_p}$ with $K/\Q_p$ ramified. Note that both $\omega_p$ and $\omega_{K/\Q_p}$ are tamely ramified. If $\omega_p^2|_{\Z_p^\times} = 1$, then the product $\omega_p \omega_{K/\Q_p}$ is unramified. By Property $(\epsilon 1)$ of epsilon factors we have 
		$\varepsilon(\omega_p \omega_{K/\Q_p}\chi_p, \phi) = \varepsilon(\chi_p, \phi)$. Also, $\varepsilon(\omega_p \omega_{K/\Q_p}, \phi) = (\omega_p \omega_{K/\Q_p})^{-1}(p)$. Now, using  \cite[Lemma 2.3]{bm} we get the desired value of $A_\theta$. The case $\omega_p^2|_{\Z_p^\times} \neq 1$ is similar to the part (1).
	\end{enumerate}
\end{proof}

\begin{theorem} \label{cp=0,1}
	Let $p$ be an odd prime and $\pi_{p} = {\rm Ind}_{W(K)}^{W(\Q_p)}(\varkappa)$ be the supercuspidal representation attached to a $p$-minimal newform $f$. If $N_p> 2$ and $C_p \leq 1$, then we have
	\begin{enumerate}
		\item 
		$\varepsilon_p =\chi_p'(e) A_\theta$ if $K/\Q_p$ is unramified.
		\item 
		If $K/\Q_p$ is ramified, then $\begin{cases}
			A_\theta, 
			\quad \text{if} \,\, (p,K|\Q_p)=1,  \\
			\Big( \frac{-1}{p} \Big)A_\theta,
			\quad \text{if} \,\, (p,K|\Q_p)=-1.
		\end{cases}$
	\end{enumerate}
	Here the number $A_\theta$ is given above and $e$ has valuation $-N_p/2$.
\end{theorem}

\begin{proof}
	By Prop. \ref{types}, only Type I representations are possible. In this case, $\varepsilon_p= \frac{\varepsilon(\varkappa^2 \chi_p', \phi_K)}{\varepsilon(\varkappa^2, \phi_K)} \times A_\theta$ by Equ. \ref{supvariance}. From the proof of Theorem \ref{supthm}, the ratio $\frac{\varepsilon(\varkappa^2 \chi_p', \phi_K)}{\varepsilon(\varkappa^2, \phi_K)} =\chi_p'(e)$ if $K/\Q_p$ is unramified; otherwise it is equal to 
	$\begin{cases}
		1, 
		\quad \text{if} \,\, (p,K|\Q_p)=1,  \\
		\Big( \frac{-1}{p} \Big),
		\quad \text{if} \,\, (p,K|\Q_p)=-1.
	\end{cases}$
	Hence, the result follows.
\end{proof}

\noindent \textbf{The case $N_p=2$:} \label{np=2} In this case, $N_p>C_p \leq 1$ and  $K/\Q_p$ is unramified with $a(\varkappa)=1$. By Equ. \ref{supvariance},
$\varepsilon_p= \frac{\varepsilon(\varkappa^2 \chi_p', \phi_K)}{\varepsilon(\varkappa^2, \phi_K)} \times A_\theta$ or $\frac{\varepsilon(\varphi \chi_p, \phi) \varepsilon(\varphi \omega_{K/\Q_p} \chi_p, \phi)}{\varepsilon(\varphi, \phi) \varepsilon(\varphi \omega_{K/\Q_p}, \phi)} \times A_\theta$
for Type I or II representations respectively. The first quotient can be computed respectively by Davenport-Hasse theorem or Stickelberger’s theorem depending upon $\circ(\widetilde{\varkappa})$ divides $(p-1)$ or $(p+1)$ as described in the proof of \cite[Theorem 5.17]{bmm}. One can compute the second quotient by the similar argument used in \cite[Lemma 5.16]{bmm}. We get the values of $A_\theta$ from Lemmas \ref{cp0} and \ref{cp1}. Combining all these, we can obtain the variance number $\varepsilon_p$ when $N_p=2$.

\smallskip

\subsection{Types of $\sym^2(\pi_p)$}
Let $M'$ denotes the prime to $p$ part of $a(\sym^2(\pi))$. In the next corollary, we classify the types of $\sym^2(\pi_p)$ where $\pi_p$ is the attached local representation at an odd prime $p$ to a $p$-minimal form $f$. Before we do that let us note down the following two properties:

\textbf{Property A:} $\varepsilon(\sym^2(\pi) \otimes \chi_p)=\chi_p(M')\varepsilon(\sym^2(\pi))$.

\textbf{Property B:} $\varepsilon(\sym^2(\pi) \otimes \chi_p)=-\chi_p(M')\varepsilon(\sym^2(\pi))$.

\begin{corollary} \label{maincoro}
	Let $f \in S_k(\Gamma_0(N), \epsilon)$ be a $p$-minimal newform with $p$ odd and let $\pi_p$ be the local representation attached to $f$ at the prime $p$. Then the types of $\sym^2(\pi_p)$ are given as follows:
	\begin{enumerate}
		\item
		$\sym^2(\pi_p)$ is a principal series representation if $N_p=C_p \geq 1$.
		\item
		$\sym^2(\pi_p)$ is a special representation if $N_p=1$ and $C_p=0$.
		\item
		Otherwise, $\sym^2(\pi_p)$ is a supercuspidal representation. In this case, $\pi_p= {\rm Ind}_{W(K)}^{W(\Q_p)}(\varkappa)$ with $K/\Q_p$ quadratic. If $N_p \geq 2$ is even, then $K/\Q_p$ is unramified. Furthermore, if $C_p\geq 2$, then %either $f$ always satisfy the Property A or 
		we have the following.
		\begin{enumerate} 
			\item 
			$\sym^2(\pi_p)$ is of Type I if the Property B holds.
			\item 
			$\sym^2(\pi_p)$ can be either of Type I or Type II if Property A holds.
		\end{enumerate}
		If $N_p \geq 4$, even with $C_p \leq 1$, then the Type II cannot occur.
		
	    \noindent If $N_p \geq 3$ is odd, then $K/\Q_p$ is ramified. Furthermore, if $C_p\geq 2$, then we have the following.
		\begin{enumerate} 
			\item [-]
			$\sym^2(\pi_p)$ is always of Type I if Property B holds.
			\item [-]
			For Type I representations, we have the following classifications of $K/\Q_p$.
			\begin{table}[h!]
				\begin{center}
					\begin{tabular}{|c|c|}
						\hline
						%$\circ(\widetilde{\omega_p})$ & $\varepsilon_p$ & $c_p$ \\
						{\bf Condition}  & {\bf  Classification of {\boldmath $K/\Q_p$}} \\
						\hline
						{Property A} &
						{$K=\Q_p(\sqrt{-p})$} \\
						\hline
						{Property B} &
						{$K=\Q_p(\sqrt{-p\zeta_{p-1}})$} \\
						\hline
					\end{tabular}
				\end{center}
			\end{table}	
			\item [-]
			If the Property A holds, then $\sym^2(\pi_p)$ can be either of Type I or II.
%			\item [-]
%			$\sym^2(\pi_p)$ cannot be of Type II when $C_p \leq 1, N_p \geq 3$ or $C_p=2, N_p=3$.
		\end{enumerate}
		$\sym^2(\pi_p)$ cannot be of Type II when $C_p \leq 1, N_p \geq 3$ or $C_p=2, N_p=3$.
	\end{enumerate}
\end{corollary}

\begin{proof}
	We claim that $\varepsilon(\sym^2(\pi) \otimes \chi_p)=\chi_p(M')\varepsilon(\sym^2(\pi)) \varepsilon_p$. Indeed, the product
	\begin{eqnarray*} 
		\prod_{p}\varepsilon_p = \varepsilon_p \prod_{q, q \neq p} \varepsilon_q 
		\overset{\text{Theorem } \ref{q}}{=}  \varepsilon_p \prod_{q, q \neq p}  \left(\frac{q}{p}\right)^{{\rm val}_q(a(\sym^2(\pi)))}.
	\end{eqnarray*}
	%For a prime $q (\neq p)$, the number $\varepsilon_q=\left(\frac{q}{p}\right)^{{\rm val}_q(a({\sym}^2(\pi)))}$ by Theorem \ref{q}. 
	Note that
	%\begin{eqnarray*}
		$\prod_{q, q \neq p} \left(\frac{q}{p}\right)^{{\rm val}_q(a({\sym}^2(\pi)))} = \chi_p(M')$.
	%\end{eqnarray*} 
Now, from the variance number, we have $\varepsilon \left(\sym^2(\pi_{p}) \otimes \chi_p \right) =  \varepsilon \left(\sym^2(\pi_{p}) \right) \varepsilon_p$. Running through all primes $p$, we get that $\varepsilon \left(\sym^2(\pi) \otimes \chi_p \right) =  \varepsilon \left(\sym^2(\pi) \right) \prod_{p}\varepsilon_p = \chi_p(M') \varepsilon \left(\sym^2(\pi) \right) \varepsilon_p$. Now, we complete the proof applying Theorem \ref{supthm} and Prop. \ref{types}.% and Lemma \ref{phi}.
\end{proof}

\subsection{The case $p=2$} We assume that the local representation attached to $f$ at $p=2$ is dihedral when it is of supercuspidal type, i.e $\pi_2= {\rm Ind}_{W(K)}^{W(\Q_2)}(\varkappa)$ with $K/\Q_2$ quadratic. Let $\delta$ denote that $2$-adic valuation of the discriminant of $K/\Q_2$. We have $\delta = 2$ or $3$. Note that $\pi=1+\sqrt{t}$ is a uniformizer of $K=\Q_2(\sqrt{t})$ when $t=-1, 3$ and $\pi=\sqrt{t}$ is a uniformizer of $K$ when $t = 2, -2, 6, -6$. We have the following. 
\begin{equation}  \label{N2}
	N_2 =
	\begin{cases}
		2 a(\varkappa), &
		\quad \text{if} \,\, K/\Q_2 \text{ is unramified},   \\
		\delta + a(\varkappa), &
		\quad \text{if} \,\, K/\Q_2 \text{ is ramified}.
	\end{cases}
\end{equation}

\begin{proposition} \label{n2}
	Let $p=2$ be a dihedral supercuspidal prime for $f$. Then we have the following.
	\begin{enumerate}
		\item 
		If $K$ is unramified, then $N_2$ is even.
		\item
		Assume $K/\Q_2$ is ramified with $a(\varkappa) \geq \delta$. 
		Then $N_2$ is odd if $\varkappa$ is minimal; otherwise $N_2$ is even.  
	\end{enumerate}
\end{proposition}

\begin{proof} See \cite[Prop. 5.10]{bm}.
 \end{proof}

%Combining Equ. \ref{N2} and Prop. \ref{n2}, we obtain the following result.
%\begin{proposition} \label{parity}
%	If $p=2$ is a dihedral supercuspidal prime for a $2$-minimal form $f$, then $a(\varkappa)$ is odd or even depending upon $\delta=2$ or $3$ respectively.
%\end{proposition}

Next, we discuss about the possibilities of the types of $\sym^2(\pi_2)$.  %Before we do that, let us recall the following useful lemma \cite[Lemma $1.8$]{tunnell} to compute the conductor of a character obtained by composing with the norm map.

\begin{proposition} \label{prop7}
	Let $f$ be a $2$-minimal form with $\pi_2= {\rm Ind}_{W(K)}^{W(\Q_2)}(\varkappa)$, $K/\Q_2$ ramified and $a(\varkappa^2) \geq \delta+1$. Then $\sym^2(\pi_2)$ cannot be of Type II if $a(\varkappa) \equiv a(\varkappa^2) \pmod 2$.  
\end{proposition}

\begin{proof}  By hypothesis $a(\varkappa) \geq \delta+1$. As $f$ is assumed to be minimal, by Prop. \ref{n2} and Equ. \ref{N2} we deduce that $a(\varkappa)$ is odd or even according as $\delta=2$ or $3$ respectively.
	
For Type II, we have $\sym^2(\pi_2)= {\rm Ind}^{W(\Q_2)}_{W(K)} (\varkappa^2) \oplus \theta = \varphi \oplus \varphi \omega_{K/\Q_2}  \oplus \theta$ for some character $\varphi$ of $\Q_2^\times$ with $\varkappa^2=\varphi \circ N_{K/\Q_2}$ and $\theta=\omega_2$ or $\omega_2\omega_{K/\Q_2}$. If $K/\Q_2$ is ramified, then by Lemma \ref{conductor norm map}, $a(\varkappa^2)=a(\varphi)+a(\varphi \omega_{K/\Q_2})-\delta$. As $a(\varkappa^2)\geq \delta+1$ by assumption, from this we conclude that $a(\varphi) \geq \delta+1$. Therefore, $a(\varkappa^2)=2a(\varphi)-\delta$  is respectively even or odd when $\delta = 2$ or $3$. This contradicts the fact $a(\varkappa) \equiv a(\varkappa^2) \pmod 2$. In other words, $\sym^2(\pi_2)$ cannot be of Type II. \end{proof}

We now compute the variance number $\varepsilon_2$ in the theorem below taking the twisting character $\chi_{-1}$ of conductor $2$, see \S \ref{epsilon}. As before, $\chi_{-1}'=\chi_{-1} \circ N_{K/\Q_2}$. Recall, $\pi$ denotes a uniformizer of $K$. Consider an element $e \in \mco_{K}^\times$ of valuation $-a(\varkappa^2)$ or $-(a(\varkappa^2)+\delta)$ depending upon $K/\Q_2$ is unramified or ramified as defined in \eqref{e}. %The theorem below determines the variance number for $p=2$.

\begin{theorem} \label{supthm2}
	Let $p=2$ be a dihedral supercupidal prime for a minimal form $f$ with $\pi_2=\Ind_{W(K)}^{W(\Q_2)} (\varkappa)$ and $K/\Q_2$ quadratic. We also assume that $C_2>3$.%If $K$ is unramified, then we have:
	\begin{enumerate}
		\item 
		If $K$ is unramified with $a(\varkappa^2)>3$, then
		\begin{equation} \label{2unramified}
			\varepsilon_2 =
			\begin{cases}
				\chi_{-1}'(e)\chi_{-1}(2)^{C_2},  & \text{ if } \,\, \mathrm{sym}^2(\pi_2) \text{ is Type I},  \\
				\chi_{-1}(2)^{C_2}, & \text{ if } \,\, \mathrm{sym}^2(\pi_2) \text{ is Type II}. %a(\varphi)>1.
			\end{cases} 
		\end{equation}
		\item 
		Let $K/\Q_2$ be ramified. %with discriminant valuation $\delta$. 
		If $a(\varkappa^2)\geq \delta+1$ and $a(\varkappa) \equiv a(\varkappa^2) \pmod 2$, then $\varepsilon_2 = \chi_{-1}(2)^{C_2}$ if $\delta=2$, and   $\chi_{-1}'(e)\chi_{-1}(2)^{C_2}$ if $\delta=3$.
	\end{enumerate}
\end{theorem}

\begin{proof}
	Let $K/\Q_2$ be unramified. Note that $a(\chi_{-1}')=2$ by Lemma \ref{conductor norm map}. As $a(\varkappa^2)>3$, using Theorem \ref{epsilon factor while twisting}, we have $\frac{\varepsilon(\varkappa^2\chi_{-1}', \phi_K)}{\varepsilon(\varkappa^2, \phi_K)}=\chi_{-1}'(e)$. On the other hand, following the proof of Lemma \ref{eps=1}, $\frac{\varepsilon(\theta \chi_{-1}, \phi)}{\varepsilon(\theta, \phi)}=\chi_{-1}(2)^{C_2}$. Thus, for Type I representations, by \eqref{supvariance} we get $\varepsilon_2 =\chi_{-1}'(e)\chi_{-1}(2)^{C_2}$. For a Type II representation, we use Lemma \ref{eps=reducible} to get the result.

	\smallskip 
	
	Let $K/\Q_2$ be ramified. By the conditions assumed, it follows from Prop. \ref{prop7} that $\sym^2(\pi_2)$ is always of Type I. Recall, $a(\varkappa)$ is odd or even according as $\delta=2$ or $3$ respectively.
	First, assume that $\delta=2$. By Lemma \ref{conductor norm map} we obtain that the induced character $\chi_{-1}'$ is unramified. As $a(\varkappa) \equiv a(\varkappa^2) \pmod 2$ by hypothesis, it follows that $a(\varkappa^2)$ is odd. Thus, using property $(\epsilon 1)$ we get that
	$\frac{\varepsilon(\varkappa^2\chi_{-1}', \phi_K)}{\varepsilon(\varkappa^2, \phi_K)}=\chi_{-1}'(\pi)^{a(\varkappa^2)-1}=1$.
	Also, as before, $\frac{\varepsilon(\theta \chi_{-1}, \phi)}{\varepsilon(\theta, \phi)}=\chi_{-1}(2)^{C_2}$. This gives the desired $\varepsilon_2$ from Equ. \ref{supvariance}.
	Next, we consider $K/\Q_2$ ramified with $\delta=3$. As $a(\omega_{K/\Q_2})=3$, by Lemma \ref{conductor norm map},  $a(\chi_{-1}')=2$. So the proof is similar to the unramified case. 
	%By Theorem \ref{epsilon factor while twisting}, we have $\frac{\varepsilon(\varkappa^2\chi_{-1}', \phi_K)}{\varepsilon(\varkappa^2, \phi_K)}=\chi_{-1}'(e)$ and as before $\frac{\varepsilon(\theta \chi_{-1}, \phi)}{\varepsilon(\theta, \phi)}=\chi_{-1}(2)^{C_2}$. Thus, by \eqref{supvariance}, $\varepsilon_2 =\chi_{-1}'(e)\chi_{-1}(2)^{C_2}$. 
\end{proof}

Now, we classify the types of $\sym^2(\pi_2)$ when $\pi_2$ is dihedral.% and it can be proved by the similar argument used in Theorem \ref{q}.
%\begin{theorem}\label{2}
%	Let $q \neq 2$. Then, $\varepsilon_q = \chi_{-1}(q)^{{\mathrm{val}}_q(a({\rm sym}^3(\pi)))}$.
%\end{theorem}

\textbf{Property C:} $\varepsilon(\sym^2(\pi) \otimes \chi_{-1}) = \chi_{-1}(M') \chi_{-1}(2)^{C_2} \varepsilon(\sym^2(\pi))$.

\textbf{Property D:} $\varepsilon(\sym^2(\pi) \otimes \chi_p) = -\chi_{-1}(M') \chi_{-1}(2)^{C_2} \varepsilon(\sym^2(\pi))$.

\begin{corollary} \label{mainkorop=2}
	Let $f$ be a $2$-minimal cusp form and $\pi_2= \Ind_{W(K)}^{W(\mathbb{Q}_2)}(\varkappa)$ be the dihedral supercuspidal representation attached to $f$ with $K/\Q_2$ quadratic. Then we have the following:
	\begin{enumerate}
		\item 
		If $N_2$ is even, then $K/\Q_2$ is unramified. In this case, if $a(\varkappa^2)>3$ then $f$ satisfies either Property C or D. Moreover, %\begin{enumerate} 
			%\item 
			$\sym^2(\pi_2)$ is of Type I if Property D holds. We cannot conclude the types if Property C holds.
			%\item 
			%$\sym^2(\pi_2)$ can be either of Type I or II if $\varepsilon({\rm sym}^2(\pi)\otimes \chi_{-1}) = \chi_{-1}(M')\chi_{-1}(2)^{C_2}\varepsilon({\rm sym}^2(\pi))$.
		%\end{enumerate}
		\item 
		If $N_2 \geq 5$ is odd, then $K/\Q_2$ is ramified. In this case, if $a(\varkappa^2)\geq \delta+1$ and $a(\varkappa) \equiv a(\varkappa^2) \pmod 2$, then ${\rm sym}^2(\pi_2)$ is always of Type I. Moreover, %if $\chi_{-1}'(e)=-1$ then 
		%\begin{enumerate} 
			%\item [-]
			%$K=\Q_2(\sqrt{t})$ with $t=-1, 3$ if $\varepsilon({\sym}^2(\pi)\otimes \chi_{-1}) = \chi_{-1}(M')\chi_{-1}(2)^{C_2}\varepsilon({\sym}^2(\pi))$.
			%\item [-]
			$K=\Q_2(\sqrt{t})$ with $t \in \{2, -2, 6, -6 \}$ if Property D holds.
		%\end{enumerate}
	\end{enumerate}
\end{corollary}
\begin{proof}
	The classification of $K/\Q_2$ depending upon the parity of $N_2$ is given in \cite[Prop. 5.10]{bm}. By the same argument used for odd primes together with Theorem \ref{supthm2}, we obtain the result.
\end{proof}

\section{Conductor of $\sym^2(\pi)$} 
In this section, we will find the conductor of $\sym^2(\pi)$ where $\pi$ is the attached representation to a minimal cusp form $f$. Suppose that $f$ satisfies the following Hypothesis {\bf (H)}: 
\begin{enumerate}
	\item [(H1)] 
	$\pi_{2}$ is dihedral when it is of supercuspidal type. 
	\item [(H2)] 
	If $\pi_p$ is induced from a ramified quadratic extension $K/\Q_p$ of a character $\varkappa$, then assume $C_p> \delta$ where $\delta$ is the $p$-adic valuation of the discriminant of $K$.
	%$C_p> \delta$ where $\delta$ is the $p$-adic valuation of the discriminant of $K/\Q_p$ with $K$ ramified.
\end{enumerate}

\begin{proposition} \label{condp>3} For an odd prime $p$, we have the following:
	\begin{enumerate}
		\item If $\pi_{p}$ is a principal series representation, then $a(\sym^2(\pi_{p}))= 2N_p$ if $N_p>1$. If $N_p=1$, we have $a(\sym^2(\pi_{p}))=1$ or $2$ depending on $\omega_p|_{\Z_p^\times}^2$ is trivial or not respectively.
		
		\item If $\pi_{p}$ is of special type, then $a(\sym^2(\pi_{p}))=2$.
		
		\item Let $\pi_p= {\rm Ind}_{W(K)}^{W(\Q_p)}(\varkappa)$ be a supercuspidal representation. If $K/\Q_p$ is unramified, then
		\begin{equation}  
			a(\sym^2(\pi_{p})) =
			\begin{cases}
				C_p, 
				\quad \quad  \quad \,  \text{if} \,\, N_p=2 \text{ with } \varkappa|_{\mathcal{O}_K^\times}^2=1,   \\
				N_p+C_p,
				\,\,  \text{ otherwise}.
			\end{cases}
		\end{equation}
	If $K/\Q_p$ is ramified, then under the hypothesis (H2) we have $a(\sym^2(\pi_{p})) = N_p+C_p$.
	\end{enumerate} 
\end{proposition} 

\begin{proof} We prove each case separately.
	\begin{enumerate} [wide, labelwidth=!, labelindent=0pt]
	\item 
	Recall that $\pi_p=\pi(\mu_1, \mu_2)$ where $\mu_1$ is unramified, $\mu_1\mu_2=\omega_p$ has conductor $N_p$. By Equ. \ref{prin}, we then deduce that $a(\sym^2(\pi_{p})) = a(\omega_p) + a(\omega_p^2)$. Thus, for odd primes $p$ with $N_p>1$, we use Prop. \ref{chisquare} to get the desired result. When $N_p=1$, the result is clear.
	
	\smallskip 
	\item  
	See \cite[\S 10, Prop.]{rohrlich}.
	
	\smallskip 
	\item If $p$ is a supercuspidal prime for $f$, then from \eqref{np} we recall that $N_p=2 a(\varkappa)$ if $K/\Q_p$ is unramified; otherwise it is $1+a(\varkappa)$.
	By Prop. \ref{symndirectsum}, the conductor of $\sym^2(\pi_{p})$ is equal to $ a\big(\Ind_{W(K)}^{W(\Q_p)} (\varkappa^2) \big) + a(\theta)$ where $\theta$ is either $\omega_p$ or $\omega_p \omega_{K/\Q_p}$. So the Equ. \ref{indconductor} implies that 
	\begin{equation} \label{eq7}
		a(\sym^2(\pi_{p})) =
		\begin{cases}
			2a(\varkappa^2) + a(\theta), \quad \quad \, \text{if} \,\, K/\Q_p \text{ is unramified}, \\
			1+a(\varkappa^2)+a(\theta), \quad \text{if} \,\, K/\Q_p \text{ is ramified}.
		\end{cases}
	\end{equation}
	
	Let $K/\Q_p$ be unramified. In this case, the quadratic character $\omega_{K/\Q_p}$ is unramified. Hence, we have $a(\theta)=a(\omega_p)=C_p$. %Using Lemma \ref{conductor norm map}, it is easy to see that  $a(\sym^2(\pi_{p}))= 2a(\varkappa^2)+C_p$. 
	If $N_p=2$ (i.e, $a(\varkappa)=1$) with $\varkappa|_{\mathcal{O}_K^\times}^2=1$, then $a(\varkappa^2)=0$; otherwise both $\varkappa$ and $\varkappa^2$ have the same conductor. Now, using \eqref{np} we obtain the desired value of $a(\sym^2(\pi_{p}))$.
	
	Now, let $K/\Q_p$ be ramified. As $f$ is assumed to be $p$-minimal, by \cite[Prop. 5.4 and Remark 5.5]{bm}, $N_p \geq 3$ is odd. So \eqref{np} implies that $a(\varkappa) \geq 2$ is even. Now since $p \geq 3$ we deduce that $a(\varkappa^2)=a(\varkappa)$. On the other hand, by (H2) we have $C_p >1$ and so the conductor of $\theta$ is $C_p$.  Hence,  $a(\sym^2(\pi_{p})) = N_p+C_p$. 
	\end{enumerate}
\end{proof}

%\begin{proposition} \label{aspecial} If $\pi_{p}$ is of special type, then $a(\sym^2(\pi_{p}))=2$. \end{proposition} \begin{proof} \end{proof}

%Define the number $c_p$:
%\begin{lemma}
%	For an odd prime $p$, let $\pi_p= {\rm Ind}_{W(K)}^{W(\Q_p)}(\varkappa)$ be a supercuspidal representation satisfying the hypothesis (H2). Then we have $a(\sym^2(\pi_p))=2N_p$.
%\end{lemma}

%\begin{proof}
	%From  Equ. \ref{symndirectsum}, the conductor of $\sym^2(\pi_{p})$ is equal to $ a\big(\Ind_{W(K)}^{W(\Q_p)} (\varkappa^2) \big) + a(\theta)$. Recall that $\theta$ is either $\omega_p$ or $\omega_p \omega_{K/\Q_p}$. By the hypothesis (H2), we have $a(\theta)=a(\omega_p)=C_p$. For $K/\Q_p$ unramified, $a(\sym^2(\pi_{p}))= 2a(\varkappa^2)+c_p=N_p+c_p$ if $N_p \neq 2$.
	
	%Let $K/\Q_p$ be ramified. $a(\varkappa) \geq 2$ is even. Hence, $a(\varkappa^2)=a(\varkappa)$. Thus, $a(\sym^2(\pi_{p}))=N_p+c_p$. 
%\end{proof}

When $\pi_2$ is of supercuspidal type, i.e, $\pi_2 = \Ind_{W(K)}^{W(\Q_2)}(\varkappa)$, let us define the following number: \begin{equation}  
	e_{\pi_2} =
	\begin{cases}
		2a(\varkappa^2), \quad \,\, \text{if} \,\, K/\Q_2 \text{ is unramified},\\
		1+a(\varkappa^2), \,\, \text{if} \,\, K/\Q_2 \text{ is ramified}.
	\end{cases}
\end{equation}

\begin{proposition} \label{p=2 sup}
	\begin{enumerate} 
	\item 
	If $\pi_{2}$ is a principal series representation, then $a(\sym^2(\pi_{2}))= 2N_2-1$ if $N_2>3$; otherwise $a(\sym^2(\pi_{2}))= N_2$. %If $N_p=1$, we have $a(\sym^2(\pi_{p}))=1$ or $2$ depending upon the order of $\omega_p$. 
	\item 
	If $\pi_{2}$ is of special type, then $a(\sym^2(\pi_{2}))=2$.
	\item 
	Let $\pi_2= {\rm Ind}_{W(K)}^{W(\Q_2)}(\varkappa)$ be a supercuspidal representation. If $K/\Q_p$ is ramified, we assume the hypothesis (H2). Then, $a(\sym^2(\pi_2))=e_{\pi_2}+C_p$.
    \end{enumerate}
\end{proposition}

\begin{proof}
	For the first part, we use Prop. \ref{chisquare} in $a(\sym^2(\pi_{2})) =$ $a(\omega_2) + a(\omega_2^2)$ to get the result. For parts (2) and (3), the proofs follow similarly as in the case of Prop. \ref{condp>3}. 
\end{proof}

\begin{defn} Let $\mathrm{S}$, $\mathrm{P}$ and $\mathrm{SC}$ denote the set of primes $p$ with $\pi_p$ respectively of special, principal series and supercuspidal type.  Let us define
	\begin{equation} \label{p1}
		\mathrm{P}_1:= \{p \in \mathrm{P}: p \text{ odd }, N_p>1 \text{ or } N_p=1 \text{ with } \omega_p|_{\Z_p^\times}^2 \neq 1 \}
	\end{equation}
	\begin{equation} \label{p2}
		\mathrm{P}_2:=  \{p \in \mathrm{P}: p=2 \text{ with } N_2>3 \}
	\end{equation}
	\begin{equation} \label{sc1}
		\mathrm{SC}_1:= \{p \in \mathrm{SC}: \pi_p= {\rm Ind}_{W(K)}^{W(\Q_p)}(\varkappa), N_p=2 \text{ with } \varkappa|_{\mathcal{O}_K^\times}^2=1\}
	\end{equation}
	\begin{equation} \label{sc2}
		\mathrm{SC}_2:= \mathrm{SC} \setminus \mathrm{SC}_1
	\end{equation}
\end{defn}

\begin{theorem} \label{conductortheorem}
	Let $f \in S_k(\Gamma_0(N), \epsilon)$ be a minimal newform satisfying the Hypothesis (H2). Let $\pi$ be the representation attached to $f$. We also assume that $p=2$ is not a supercuspidal prime. Then
	\begin{equation} \label{conductorthm}
		a(\sym^2(\pi))=N ~~ \prod_{p \in \mathrm{S}} p ~~ \prod_{p \in \mathrm{P}_1 \cup \mathrm{SC}_2} p^{C_p}~~ \prod_{p \in \mathrm{P}_2} p^{C_p-1}~~\prod_{p \in \mathrm{SC}_1}  p^{C_p-2}.
	\end{equation}
\end{theorem}

\begin{proof}
	Note that $a(\pi)=N=\prod_p p^{N_p}$. Using Propositions \ref{condp>3} and \ref{p=2 sup}, we obtain that \[a(\sym^2(\pi))=N ~~ \prod_{p \in \mathrm{S}} p ~~ \prod_{p \in \mathrm{P}_1} p^{N_p}~~ \prod_{p \in \mathrm{P}_2} p^{N_p-1}~~\prod_{p \in \mathrm{SC}_1}  p^{C_p-2} ~~\prod_{p \in \mathrm{SC}_2} p^{C_p}.\]
	Noting $N_p=C_p$ for primes with $\pi_p$ ramified principal series type, we complete the proof. %of the theorem.
\end{proof}

\begin{remark} 
	In the setting of the above theorem, we now assume the situations that are not considered before.  These are given as follows: 
	\begin{enumerate}[wide, labelwidth=!, labelindent=0pt]
		%\smallskip 
		\item 
		If $p=2$ is a supercuspidal prime for $f$ then we assume the Hypothesis \textbf{(H)}. As $\pi_2$ is a dihedral supercuspidal representation, from Prop. \ref{p=2 sup} we have that $a(\sym^2(\pi_2))=e_{\pi_2}+C_p$. Thus, the extra term $2^{e_{\pi_2}+C_p-N_p}$ will be involved in the formula \ref{conductorthm}. 
		%\smallskip 
		\item	
		 Without Hypothesis (H2). For odd primes $p$, recall that the Hypothesis (H2) implies $\pi_p$ is induced from a quadratic ramified extension and $C_p>1$. In such cases, Prop. \ref{condp>3} determines that $a(\sym^2(\pi_{p})) =1+a(\varkappa^2)+a(\theta)=N_p+C_p$ where $\theta$ is either $\omega_p$ or $\omega_p \omega_{K/\Q_p}$. Assume that $C_p \leq 1$. For $C_p=0$, we have either $a(\theta)= 0$ or $1$ depending upon $\theta=\omega_p$ or $\omega_p \omega_{K/\Q_p}$ respectively. If $C_p=1$, then we have $a(\theta) =
		\begin{cases}
		 	0, \quad \text{if} \,\, \theta= \omega_p \omega_{K/\Q_p}  \text{ and }(\omega_p|_{\Z_p^\times})^2=1,\\
		 	1, \quad \text{otherwise}.
		 \end{cases}$ Denoting $a(\theta)$ by $\widetilde{C}_p$, for $C_p \leq 1$ we therefore obtain that $a(\sym^2(\pi_{p})) = N_p+ \widetilde{C}_p$. For such values of $C_p$, the product $\prod_{p \in \mathrm{SC}_2} p^{C_p}$ in the above formula \ref{conductorthm} is  then replaced by $\prod_{p \in \mathrm{SCU}} p^{C_p} ~~ \prod_{p \in \mathrm{SCR}} p^{\widetilde{C}_p}$, where $\mathrm{SCU}$ and $\mathrm{SCR}$ are two subsets of $\mathrm{SC}_2$ such that $\pi_p$ is respectively induced from an unramified and a ramified quadratic extension of $\Q_p$.
	\end{enumerate}
\end{remark}

\end{document}